\documentclass[12pt]{amsart}
\usepackage[margin=1in]{geometry}
 \usepackage{amsmath,amsfonts,amsthm,amssymb,bbm}
\usepackage{graphicx,color,dsfont}
\usepackage{enumitem}
\usepackage{fourier}

\newtheorem{theorem}{Theorem}
\newtheorem{lemma}{Lemma}
\newtheorem{proposition}{Proposition}

\newtheorem{corollary}{Corollary}

\newtheorem{claim}{Claim}

 \newtheorem{assumption}{Assumption}

 \theoremstyle{definition}
 
 \theoremstyle{remark}

 \numberwithin{equation}{section}

\newcommand{\vertiii}[1]{{\left\vert\kern-0.25ex\left\vert\kern-0.25ex\left\vert #1
    \right\vert\kern-0.25ex\right\vert\kern-0.25ex\right\vert}}

\newcommand{\R}{{\mathbb R}}

\newcommand{\f}[2]{\frac{#1}{#2}}

\newcommand{\cl}{{\mathcal L}}

\newcommand{\al}{\alpha}
\newcommand{\be}{\beta}

\newcommand{\ga}{\gamma}
\newcommand{\Ga}{\Gamma}
\newcommand{\de}{\delta}

\newcommand{\De}{\Delta}

\newcommand{\ka}{\kappa}
\newcommand{\la}{\lambda}

\newcommand{\si}{\sigma}

\newcommand{\rn}{{\mathbb R}^n}
\newcommand{\rne}{{\mathbb R}^{n-1}}

\newcommand{\rmm}{{\mathbb R}^m}
\newcommand{\rone}{\mathbb R}
\newcommand{\rtwo}{\mathbf R^2}

\newcommand{\dpr}[2]{\langle #1,#2 \rangle}

\newcommand{\eps}{\epsilon}

\newcommand{\cq}{\mathcal Q}
\newcommand{\cp}{\mathcal P}

\newcommand{\p}{\partial}

\newcommand{\beq}{\begin{equation}}
\newcommand{\eeq}{\end{equation}}
\newcommand{\beqna}{\begin{eqnarray*}}
\newcommand{\eeqna}{\end{eqnarray*}}
\newcommand{\beqn}{\begin{equation*}}
\newcommand{\eeqn}{\end{equation*}}
\newcommand{\bp}{\begin{proof}}
\newcommand{\ep}{\end{proof}}
\newcommand{\bprop}{\begin{proposition}}
\newcommand{\eprop}{\end{proposition}}
\newcommand{\bt}{\begin{theorem}}
\newcommand{\et}{\end{theorem}}
\newcommand{\bex}{\begin{Example}}
\newcommand{\eex}{\end{Example}}
\newcommand{\bc}{\begin{corollary}}
\newcommand{\ec}{\end{corollary}}
\newcommand{\bcl}{\begin{claim}}
\newcommand{\ecl}{\end{claim}}
\newcommand{\bl}{\begin{lemma}}
\newcommand{\el}{\end{lemma}}

\begin{document}

\title[Sharp relaxation rates  for plane waves]
 {Sharp relaxation rates for  plane waves of  reaction-diffusion systems}

\thanks{Stefanov's research  is partially  supported by  NSF-DMS   under grant  \# 1614734.}

 \author[Fazel Hadadifrad]{\sc Fazel Hadadifard}
 \address{ Department of Mathematics,
 	University of Kansas,
 	1460 Jayhawk Boulevard,  Lawrence KS 66045--7523, USA}
 \email{f.hadadi@ku.edu}
 
 \author[Atanas Stefanov]{\sc Atanas G. Stefanov}
 \address{ Department of Mathematics,
 	University of Kansas,
 	1460 Jayhawk Boulevard,  Lawrence KS 66045--7523, USA}
 \email{stefanov@ku.edu}



\date{\today}

\begin{abstract}
 It is well-known and classical result that  spectrally  stable traveling waves of a general reaction-diffusion system in one spatial dimension are asymptotically stable with exponential relaxation rates. In a series of works in the 1990's, \cite{Good, TK, LX, Xin}, the authors have considered plane traveling waves for such systems and they have succeeded in showing  asymptotic stability for such objects. Interestingly, the   (estimates for the) relaxation rates that they have exhibited, are all algebraic and dimension dependent. It was heuristically argued that as the spectral gap closes in dimensions $n\geq 2$, algebraic rates are the best possible. 
 
 In this paper, we revisit this issue. We rigorously calculate the  sharp relaxation rates in $L^\infty$ based  spaces, both for the asymptotic phase and the radiation terms.  These turn out to be   are indeed algebraic, but about twice better than the best ones obtained in these early works,  although this can be mostly  attributed  to the inefficiencies of using  Sobolev embeddings to control $L^\infty$ norms by high order $L^2$ based Sobolev space norms.  Finally, we explicitly construct the leading order profiles, both for the phase and the radiation terms. Our approach relies on the method of scaling variables, as introduced in  \cite{GW, GW1} and in fact provides sharp relaxation rates in a class of weighted $L^2$ spaces as well.

\end{abstract}

\maketitle

\section{Introduction}
 In this paper, we study   the following general reaction-diffusion models   
 \begin{eqnarray}
 \label{RD}
 \left\{ 
 \begin{array}{l}
 u_t= \De u+ f(u),~~~~~~ x \in \rn \\ 
 u(0)= u_0,
 \end{array}  
 \right.
 \end{eqnarray}
   where, $n \geq 2$,  $u:  \rn \times \mathbb{R}^+ \rightarrow \rmm$, $m\geq 1$,  and 
   $f\in C^4(\rn, \rmm)$.  More precisely, we will be interested in the dynamics of the solutions with initial data close to   plane waves, that is the dynamics near plane waves. Existence and stability of such waves in the case $n=1$ is a classical subject, with a vast literature associated to it. 
   
   In order to introduce the problem and some notations,  assume that there exist steady states $\phi_{\pm}\in \rmm$,   so that $f(\phi_{\pm})=0$. Next, we assume that $n=1$ and there exists solutions of \eqref{RD}, in the form $u(t,x)=\phi(x-ct)$.  That is, $\phi$ satisfies the 
   one-dimensional profile equation, 
   \begin{equation}
   \label{10} 
   \phi''(z)+c\phi'(z)+f(\phi(z))=0, z\in \rone.
   \end{equation}
   We also assume that $\lim_{z\to \pm \infty} \phi(z)=\phi_\pm$, with exponential rates of convergence, although the exponential rate of convergence can be replaced with a weaker, but nevertheless strong enough algebraic rate. In any case, our standing assumption is 
    that  for some $\upsilon>0$, there is 
   $$
   |\phi(z)-\phi_-|\leq C e^{\upsilon z}, z<0; \ \ |\phi(z)-\phi_+|\leq C e^{-\upsilon z}, z>0
   $$
   Finally, we assume that the localized function $\phi': \phi'\in H^2(\rone)$. Another relevant object  for  the stability theory is the  (one-dimensional)  linearized operator about the wave, namely 
   $$
   L_1= \partial_{zz}+  c \partial_z+ D f(\phi), \ \ D(L_1)=H^2(\rone). 
   $$
    Saying that $\phi$ is spectrally stable amounts to $\si(L_1)\subset {\mathbb C}_-=\{\la: \Re\la\leq 0\}$. Very often, waves like that enjoy the strong spectral stability property, namely that\footnote{Here observe that $0$ is automatically in the spectrum as corresponding to a translational invariance or just by virtue  of  taking $\p_z$ in the profile equation \eqref{10}. } $\si(L_1)\subset \{0\}\cup \{\la: \Re\la\leq -\de\}$ for some $\de>0$. 
    It is a classical result by now that for the $n=1$ problem $u_t=u_{xx}+f(u)$ such solutions are asymptotically stable, \cite{Henry, Sat}, and in fact they enjoy exponential relaxation rates. 
    
    The situation becomes more interesting for the case of plane waves. We now introduce the notion of  plane wave solutions.    These are in the form  $u(t,x)=\phi(\ka\cdot x-ct)$, where $\ka\in \mathbb{S}^{n-1}$. It is clear that $\phi$ satisfies the same     one-dimensional profile equation, \eqref{10}. 
    In fact, without loss of generality, we may assume that $\ka=(1, 0, \ldots, 0)$ as the problem is rotationally invariant. These solutions $\phi$, if they exist, are referred to as {\it plane waves}. Since all  statements we make for traveling plane waves in the form $\phi(x_1-ct, x_2, \ldots, x_n)$ will be easily translatable for general plane  waves of the form $\phi(\ka\cdot x-ct)$ for arbitrary $\ka\in \mathbb{S}^{n-1}$, we henceforth concentrate on the case of waves in the form $\phi(z-ct, x_2, \ldots, x_n)$.  Passing to the moving frame of reference $x_1-c t\to z$ renders the equation \eqref{RD} in the form 
    \begin{equation}
    \label{12} 
    u_t=\De u+ c \p_{z} u +f(u), x \in \rn, 
    \end{equation}  
    To reiterate, going forward, we consider stationary solutions of \eqref{12}, instead of traveling waves for \eqref{RD}. This is, as discussed above, an equivalent problem. 
    
    The study of the plane waves and their stability has attracted a lot of interest over the last thirty years.  The following,  very incomplete,   list   \cite{BJ, BKSS, GLS, GL, TK, TK2, LX, LW, LWang, Sat, TZKS, Xin}, consists of mostly recent references as well as various applications to the sciences. 
    
    We have already mentioned about  asymptotic stability for these waves, so it is time for some rigorous introductions.     More specifically, asymptotic stability in this context  means  that for any initial  data $u_0$, close to the plane wave $\phi$ in an appropriate norm, there is an asymptotic phase $\si(t, y),  x=(z, y)$,  so that  the  radiation term tends to zero, i.e. 
  \begin{equation}
  \label{40}
   \lim_{t\to \infty} \|u(t, z,y)-\phi(z - \si(t,y))\|_{X}=0,  
  \end{equation}
  for  some appropriate function space $X$ in the variables $(z,y)\in \rone\times \rne$. 
  It should be mentioned that the introduction of a  $(t,y)$ dependent asymptotic phase $\si$ is absolutely necessary in order for an estimate like \eqref{40} to hold true. See for example Remark $1.3$ in \cite{TK}. 
  
 Regarding specific results about asymptotic stability of plane waves, let  us begin by   stating that the general  question has been resolved,  for the generality that we are interested in,  in a very satisfactory fashion,  in the works \cite{Good, TK, LX,  Xin}.  Subsequently, and in a more general context in \cite{GLS, GL,  LW, LWang, TZKS},  For some of these later results, the authors consider degenerate systems appearing in certain combustion and biological applications, where the spectral gap property fails even in one spatial dimension.  These works necessitates the introduction of exponentially weighted spaces to effectively create such spectral gap, but this will be outside the scope of this paper. We shall instead concentrate on the easier and yet not very well-understood case, where we start with a spectral gap in one spatial dimension, i.e. the setup in \cite{TK, LX,  Xin}. 
 
   In order to summarize the state of the art, the results in these papers establish that as soon as $n\geq 2$, there is an  {\it algebraic in time estimate} for the relaxation rates in various Sobolev norms.  This is indeed in sharp contrast with the case of one spatial dimension, where under the same spectral assumptions (see the discussion below Assumption \ref{assumption}), one can show, see \cite{Henry, Sat},  that both the radiation and the phase go to zero at an exponential rate. 
   \subsection{Linearized operators} 
Let us introduce the full   linearized operator for the plane wave that arises.  Namely
   \begin{eqnarray*}
   &&L= \De+ c \partial_z+ D f(\phi)=L_1+\De_y, D(L)=H^2(\rn).
   \end{eqnarray*}
  Clearly,  $L$ is a  closed operator.   Due to our assumptions, $\phi$ is a bounded function, whence $L$ is a (non self-adjoint)  Schr\"odinger operator with a drift term. It is a classical fact that for the related one dimensional operator, we have  $L_1[\phi']=0$, which is obtained by differentiating the profile equation \eqref{10} in $z$. This is of course nothing but a manifestation of the fact that the problem is translationally invariant and hence zero is an eigenvalue. 
   As we have alluded to above, the spectral stability of the wave $\phi$, as a  solution to the one dimensional model \eqref{12}, consist in the fact that $\si(L_1)\subset \{z: \Re z\leq 0\}$. Moreover, we shall need to require that in fact its spectrum is a fixed distance $\de>0$ away from the marginal axes $\Re z=0$, except for the translational eigenvalue at zero, which we assume to be simple. More specifically, we make the following standing assumption henceforth. 
   \begin{assumption}
   	\label{assumption} We assume that there exists $\de>0$, so that the spectrum of $L_1$ in $H^1(\rone)$ satisfies 
   	\begin{equation}
   	\label{20} 
   	\sigma(L_1)	\setminus \{0\} \subset \{\la: \Re \la \leq - \de\}
   	\end{equation}
   	Moreover, the eigenvalue at zero is simple, with an eigenfunction $\phi'$. 
   \end{assumption}
   Having the spectral gap condition \eqref{20}, and under appropriate conditions on $f, \phi$, allows one  to  show  that the wave $\phi$ is asymptotically stable, with exponential decay of the radiation term, with an exponential rate of essentially $e^{-(\de-\eps) t}$.   This goes back to at least the classical works \cite{BJ, Henry}. In the case of plane waves, one has $L$ instead of $L_1$ as a linearized operator, which destroys the spectral gap property. In fact,  since $L=L_1+\De_y$,  a direct computation shows that 
   $$
   L[\phi'(z) e^{i k \cdot y}]=-k^2 \phi'(z) e^{i k \cdot y},
   $$
   whence it becomes immediately clear that  the continuous spectrum of $L$ contains the whole negative real axes. In particular, it touches the imaginary axes at zero, so that the corresponding semigroup $e^{t L}$ has at best polynomial rate of decay. Heuristically, one expects no better from the nonlinear problem, so polynomial in time bounds seem  indeed the best possible in \eqref{40}. 
   
   {\it This is however an open problem, and one of the goals of this paper is to establish this rigorously.     In fact, we aim at establishing  the optimal decay rates in these asymptotic results}. We achieve that by requiring  slightly more localized initial  perturbations $v_0:=u_0-\phi$, namely that $v_0$ resides in an appropriate (power) weighted $L^2$ space, see Section \ref{sec:1.2} below.  Before we state our concrete results, let us discuss the setup of the asymptotic stability result. This part follows the work of  Kapitula, \cite{TK}, but note that we introduce weighted spaces for the purposes of our analysis later on. 
   \subsection{Setup  of the asymptotic profile  equations} 
   \label{sec:1.2}
 We start with the Riesz projection for $L_1$, associated with the isolated and simple eigenvalue at zero. Namely, for a small $\eps$,  introduce 
 \begin{equation}
 \label{50} 
 P_0 u= \f{1}{2 \pi i} \int_{|\la|=\eps} (\la- L_1)^{-1} d\ \la
 \end{equation} 	
  As zero is a simple eigenvalue, with an eigenfunction $\phi'$, it follows by the Riesz representation theorem\footnote{In this work, we only use real-valued functions, so the dot product is symmetric $\dpr{\psi}{u}=\dpr{u}{\psi}$} that for $u\in L^2(\rone)$, $P_0 u=\dpr{\psi}{u} \phi'$, where $\psi\in H^2(\rone)$ and in fact $L^*\psi=0$, with the normalization, $\dpr{\psi}{\phi'}=1$, see \cite{TK2}. In addition, we define $Q_0=Id-P_0$, and both operators commute with $L_1$. While the operators $P_0, Q_0$ act upon functions of the first variable only, we may also consider their  action on functions, which depend on the remaining variables $t,y$ as well. 
  
  Introduce   the weighted $L^2(m)(\rne)$ spaces, or $L^2(m)$ for short, as follows 
  $$
  L^2(m)=\left\{f: \rne\to\rone:  \|f\|_{L^2(m)}=\left(\int_{\rne} (1+|y|^2)^m |f(y)|^2dy \right)^{1/2}<\infty   \right\}. 
  $$
  for some $m\in \rone $.  Also, define 
  $$
  H^1(m) =\{f: \rne\to\rone: f, \nabla_y f\in L^2(m)  \}. 
  $$
  Note that all the spaces in this section are based on functions on  $\rne$, due to the fact that $y\in\rne$. In anticipation of our analysis later, we introduce the spaces $(H^1(m)\cap W^{1, \infty})_y H^1_z$ for functions $f(y,z)$, where the norm is taken as follows 
  $$
  \|f\|_{(H^1(m)\cap W^{1, \infty})_y H^1_z}^2 =\sum_{a,b\in \{0,1\}} [\int_{\rn } |\nabla_z^a \nabla^b_y f(y,z)|^2 (1+|y|^2)^m dy dz + \sup_{y\in \rne}  \|\nabla^a_z \nabla^b_y f(z,y)\|_{L^2_z}^2 ] 
  $$
  As is clear from the definition above,  we shall adopt the notion that all  norms in the $z$ variable shall be always  taken first. 
  Introduce 
  the complementary subspaces 
  \begin{eqnarray*}
  &&\mathcal{N}= \{u \in (H^1(m)\cap W^{1, \infty})_y H^1_{z}  : u= P_0u \}\\  
  &&\mathcal{R}= \{u \in (H^1(m)\cap W^{1, \infty})_y H^1_{z}: u= Q_0u \}.
  \end{eqnarray*}
  Clearly $  (H^1(m)\cap W^{1, \infty})_y H^1_{z}   = \mathcal{N}+\mathcal{R}$, in the sense that every function in the base space\footnote{Here, we would like to note that our base space  is a bit different than the one used by the previous authors, who preferred to use high order Sobolev spaces, which control $L^\infty(\rn)$.} \\ $  (H^1(m)\cap W^{1, \infty})_y H^1_{z}$  is uniquely representable as a sum of two functions in $\mathcal{N}$ and $\mathcal{R}$ respectively.  We need the following lemma\footnote{see Lemma 2.2 in \cite{TK} for a similar statement, in high order Sobolev spaces.}
  \begin{lemma}
  	\label{le:10} 
  There exists   $\eps_0>0$ and a constant $C$, so that for all  
  $w:  \|w\|_{(H^1(m)\cap W^{1, \infty})_y H^1_{z}}<\eps_0$, one can find unique and small $(v(w), \si(w))\in \mathcal{R}\times H^1(m) \cap W^{1, \infty}$, so that 
  $$
  \|v(w)\|_{(H^1(m)\cap W^{1, \infty})_y H^1_{z}}+\|\si(w)\|_{H^1(m) \cap W^{1, \infty} }<C \eps_0
  $$ 
  and 
  \begin{equation}
  \label{70}
   \phi(z) + w(z, y) = \phi(z- \sigma(y))+ v(z, y).
  \end{equation}
  \end{lemma}
  The proof of the lemma involves a  standard application of the implicit function theorem, but its   proof  will be presented in the Appendix for completeness. Note that we can apply Lemma \ref{le:10} and in particular decomposition \eqref{70} for time dependent perturbations, so long as the smallness condition is satisfied. 
  
  Using the ansatz provided by  \eqref{70}, and as long as 
  $\|w(t, \cdot)\|_{(H^1(m)\cap W^{1, \infty})_y H^1_{z}}<<1$,   the equation   \eqref{12} is transformed  into the following system of equations
  \begin{eqnarray}
  \label{vv}
  \left\{ 
  \begin{array}{l}
  v_t= L v+ Q_0 H(\phi_{\sigma}, v)+ Q_0 N_1(\sigma, \nabla_y\cdot \sigma, v)\\
  \sigma_t= \De_y \sigma+ N_2(\sigma, \nabla_y \cdot \sigma, v),\\
  v(0)= v_0, \ \ 
  \sigma(0)= \sigma_0
  \end{array}  
  \right.
  \end{eqnarray}
  where $\phi_\si(z):=\phi(z-\si(t,y))$ and\footnote{Here $D^2f(\phi_{\sigma}) v^2$ is a quadratic form and it denotes  the action of the Hessian matrix $D^2f(\phi_{\sigma})$ on $(v,v)$. We will use the same convention later on for trilinear  forms}
  \begin{eqnarray*}
  && H(\phi_{\sigma}, v)= f(v+ \phi_{\sigma})- f(\phi_{\sigma})- Df(\phi_{\sigma}) =:\f{1}{2} D^2f(\phi_{\sigma}) v^2+ E(v)\\
  && N_2(\sigma, \nabla_y\cdot  \sigma, v)= K_1(\sigma) (\nabla_y \cdot \sigma)^2+ K_2(\sigma) \bigg(\langle \psi, H(\phi_{\sigma}, v) \rangle+ (D f(\phi_{\sigma})- D f(\phi)) v \rangle\bigg) \\
  && N_1(\sigma, \nabla_y \cdot  \sigma, v)=  N_2(\sigma, \nabla_y\cdot  \sigma, v) \phi'_\si+ \big(D f(\phi_{\sigma})- D f(\phi)\big) v+ (\nabla_y \cdot \sigma)^2 \phi''_{\sigma}\\
  &&K_1(\sigma)= - \f{\langle \psi, \phi''_{\sigma}\rangle}{\langle \psi, \phi'_{\sigma} \rangle}, \ \ K_2(\sigma)=   \f{1}{\langle \psi, \phi'_{\sigma} \rangle}.
  \end{eqnarray*}
  This is in fact done in great details then in \cite{TK}, see equations $(2.28),  (2.29)$ on p. 261 there. One of the important points, \cite{TK},  is that with $\|\si\|_{L^\infty}<<1$ guaranteed by Lemma \ref{le:10}, we have that 
  $\langle \psi, \phi'_{\sigma} \rangle=\dpr{\psi}{\phi'}+\dpr{\psi}{\phi'_{\sigma}-\phi'}=1+O(\si)$, whence the denominators in the coefficients $K_j(\si), j=1,2$ are away from zero. 
    
     The error term is of the form
  \begin{equation}
  \label{100} 
  E(v)=f(v+ \phi_{\sigma})- f(\phi_{\sigma})- Df(\phi_{\sigma}) v -\f{1}{2} D^2f(\phi_{\sigma}) v^2=O(v^3),
  \end{equation}
  under the assumption $f\in C^3(\rone)$ and $\phi$ is a bounded function. We provide further concrete estimate on $E(v)$ later on, where we shall need to assume $f\in C^4$, since spatial derivatives on $E$ need to be taken. See the proof of Lemma \ref{nonl} below. 
  \subsection{Main results}
  
  As we have already discussed, in this paper we  provide the sharp time decay rate for $\sigma$ and $v$  in \eqref{vv}.  The following theorems are our main results. 
  \begin{theorem}
  	\label{theo:10} Let $n\geq 2$ and $m>\f{n}{2}+1$.  There exists small $\eps_0>0$ and a constant $C$, so that the stationary solutions of \eqref{12}are asymptotically stable. More precisely, for all $\eps: 0<\eps<\eps_0$ and for all  
  	$u_0: \|u_0(z,y)-\phi(z)\|_{(H^1(m)\cap W^{1, \infty})_y H^1_{z}}<\eps$, the solution to \eqref{12} with initial data $u_0$ is global and there exists $\si\in L^\infty(\rone, (H^1(m)\cap W^{1, \infty}))$, so that 
  	$$
  	u(t,z, y)=\phi(z- \si(t,y))+v(t,z,y), \ \ v=Q_0v \in  L^\infty(\rone, (H^1(m)\cap W^{1, \infty})_y H^1_z)
  	$$
  	with 
  	 \begin{eqnarray}
  	 \label{1.10} 
  	 	& &  \|\si(t, \cdot)\|_{L^\infty_y}\leq C \eps (1+t)^{-\f{n-1}{2}} \\
  	 	 \label{1.11} 
  	 	& & \|\nabla_y \si(t, \cdot)\|_{L^\infty_y}\leq C \eps (1+t)^{-\f{n}{2}}\\
  	 	 \label{1.12} 
  	 	& & \|v\|_{L^\infty_{y,z}}\leq C\eps (1+t)^{-(n+\f{1}{2})} 
  	 \end{eqnarray}
  \end{theorem}
  {\bf Remarks:} 
  \begin{itemize}
  		\item The estimates for $v$ can be stated in a more precise form as follows 
  		$$
  		\|v\|_{L^\infty_{y,z}}\leq C(\eps^2 (1+t)^{-(n+\f{1}{2})}+\eps e^{-\f{\de}{2}t}),
  		$$
  		of which \eqref{1.12} is a corollary. In other words, there are two terms in the formula for $v$ - one linear in $\eps$, but decaying exponentially in $t$ (coming from free solutions), while the other decaying at the right power rate, but quadratic in $\eps$, which comes from the Duhamel's term and the nonlinearity respectively.  
  	\item The decay estimates in $L^\infty_{y z}$ norms \eqref{1.10}, \eqref{1.12} should be compared with the estimates in \cite{Xin}, \cite{TK}. As the arguments in these papers require the use of Sobolev embedding into $H^{k}$ spaces, it only provides the bound $\|\si\|_{L^\infty}\leq C \eps (1+t)^{-\f{n-1}{4}}$, whereas \eqref{1.10} is clearly much better. In fact, \eqref{1.10} is sharp, as shown in Theorem \ref{theo:25} below. The estimate \eqref{1.12}  for $v$ above is also clearly superior to the one provided in \cite{TK}. 
  
  	\item We have more estimates for $\si, v$ than the one stated in Theorem \ref{theo:10}. In particular, $v, \si$ belong to weighted $L^2$ spaces and in fact, one can write   estimates as follows - for every $0\leq \tilde{m}\leq m$, 
$$  
\left(\int_{\rne} |\si(t,y)|^2 |y|^{2\tilde{m}}  dy \right)^{1/2} \leq C\eps 
(1+t)^{-\f{1}{2}( \f{n-1}{2}-\tilde{m})},  
$$
  This estimate gives an  algebraic decay for $\tilde{m}<\f{n-1}{2}$,  but they are true even if $\tilde{m}$ is larger, that is the corresponding weighted $L^2$ norms may be growing in $t$. In  the case $\tilde{m}=0$, these become the usual $L^2$ spaces. One can in fact see that the result, in this case exactly matches the $L^2$ bounds in \cite{TK}. 
  \item One disadvantage of our method is that one cannot get estimates for $\nabla^2_y \si$ nor $\nabla^2_y v$ (and higher order derivatives), due to a technical issue that arises in the scaled variable analysis, see the remark after Proposition \ref{prop:14} below.  Such estimates are clearly possible, as was demonstrated in \cite{TK}.   On the other hand, we believe that this is really a technical issue, which we have not explored further. 
  \end{itemize}
  	 The rates established in Theorem \ref{theo:10}  are sharp. Specifically, we have the following result, which we formulate as a separate theorem. 
  	 \begin{theorem}
  	 	\label{theo:25} 
  	 	Under the assumptions of Theorem \ref{theo:10}, the estimates \eqref{1.10}, \eqref{1.11} and \eqref{1.12} are sharp. More precisely,   	let $u_0: \|u_0(y,z)-\phi(z)\|_{(H^1(m)\cap W^{1, \infty})_y H^1_{z}}<\eps$ and $\si_0 \in H^1(m)\cap W^{1, \infty}$,  $v_0=Q_0 v_0  \in (H^1(m)\cap W^{1, \infty})_y H^1_z$ be the unique pair guaranteed by Lemma \ref{le:10}, so that 
  	 	$$
  	 	u_0(y,z)=\phi(z-\si_0(y))+v_0(z,y). 
  	 	$$
  	 	Then, we have the following 	
  	 		\begin{eqnarray}
  	 		\label{71} 
  	 		& &   \left\|\si(t, \cdot)-  \f{(\int_{\rne} \si_0(y) dy)}{(1+t)^{\f{n-1}{2}}}  G\left(\f{\cdot}{\sqrt{1+t}}\right) \right\|_{L^\infty_y}\leq \f{C\eps^2}{(1+t)^{\f{n}{2}}}, \\
  	 			\label{72} 
  	 		& &   \left\|\p_j \si(t, \cdot)-  \f{(\int_{\rne} \si_0(y) dy)}{(1+t)^{\f{n}{2}}}  
  	 		(\p_jG)\left(\f{\cdot}{\sqrt{1+t}}\right) \right\|_{L^\infty_y}\leq \f{C\eps^2}{(1+t)^{\f{n+1}{2}}}, j=1, \ldots, n-1,
  	 		\end{eqnarray}
  	 		where $G(y)= (4\pi)^{-\f{n-1}{2}} e^{-\f{|y|^2}{4}}$. In particular, assuming that $\int_{\rne} \si_0(y) dy\neq 0$, we have the asymptotics 
  	 		$$
  	 		\|\si(t, \cdot)\|_{L^\infty_y}\simeq \eps (1+t)^{-\f{n-1}{2}}, \ \ 	\|\nabla \si(t, \cdot)\|_{L^\infty_y}\simeq \eps (1+t)^{-\f{n}{2}}
  	 		$$
  	 		Regarding $v$, we have that for\footnote{note that $L_1$ is invertible on $Q_0[L^2_z]$ or $L_1^{-1} Q_0$ is well defined}  $n\geq 3$, 
  	 	\begin{equation}
  	 	\label{831} 
  	 		 \|v(t,z,y) + \f{(\int_{\rne} \si_0(y) dy)^2}{(4\pi)^{n-1}} \f{e^{-\f{|y|^2}{2(t+1)}}}{(t+1)^{n+\f{1}{2}}} L_1^{-1} Q_0[\phi''] (z) \|_{L^\infty_{z,y}}\leq C(\eps^2 (1+t)^{-n-1}+\eps e^{-\f{\de}{2}t}). 
  	 	\end{equation}	
  	 	whereas for $n=2$, 
  	 	\begin{equation}
  	 	\label{832} 
  	 	\|v(t,z,y) + \f{(\int_{\rone} \si_0(y) dy)^2}{4\pi} \f{e^{-\f{|y|^2}{2(t+1)}}}{(t+1)^{\f{5}{2}}} L_1^{-1} Q_0[\phi''] (z) \|_{L^\infty_{z,y}}\leq C(\eps^3 (1+t)^{-\f{5}{2}}+ \eps^2 (1+t)^{-3}+\eps e^{-\f{\de}{2}t}). 
  	 	\end{equation}	
  	 	In particular, if $\int_{\rne} \si_0(y) dy\neq 0$, we have the asymptotics  
\begin{equation}
\label{862} 
	\|v(t, \cdot)\|_{L^\infty_{y,z}}\simeq \eps^2 (1+t)^{-n-\f{1}{2}}. 
\end{equation}
  	 	\end{theorem}
  	 	{\bf Remarks:} 
  	 	\begin{itemize}
  	 		\item The asymptotic expansion for $\si$ improves both in the order of $\eps$ and the decay rate - the leading order term is order $\eps (1+t)^{-\f{n-1}{2}}$, while the error is $\eps^2 (1+t)^{-\f{n}{2}}$. This is due to the fact that the leading order term entirely originates  from the free solution.
  	 		\item  In contrast, the expansion for $v$ has a main term, which is $\eps^2 (1+t)^{-n-\f{1}{2}}$ and two to three types of error terms - an exponentially decaying in $t$, but linear in $\eps$ (originating from initial data) and faster decaying, but still quadratic in $\eps$ terms, originating from various other nonlinear terms. In the case $n=2$, we recover yet another term, which decays like the main term, but it is order of $\eps$ smaller. Most importantly, the structure of the error terms guarantees \eqref{862}. 
  	 	\end{itemize}
  \section{Preliminary steps} 
  In this section, we transform the evolution equation \eqref{vv} into an equivalent one, through the use of the so-called scaling variables. 
  
  \subsection{The evolution system in scaling variables} Introduce  the scaling variables 
  $$
  \tau= \ln(1+ t), \ \ \eta_j= \frac{y_j}{\sqrt{1+ t}}, j=2, \ldots, n. 
  $$
In  these independent variables, set the new dependent variables $V, \Ga$ as follows 
$$
v(z, y, t)= \frac{1}{{1+ t}} V\left(z, \frac{y}{\sqrt{1+ t}},  \ln(1+ t)\right), \ \ \sigma(y, t)= 
\frac{1}{\sqrt{1+ t}} \Ga \left(\frac{y}{\sqrt{1+ t}}, \ln(1+ t)\right).
$$
 Straightforward computations show 
   \begin{eqnarray*}  
   v_t&=& - \f{1}{(1+ t)^2} V- \f{1}{2} \f{1}{(1+ t)^2} \f{y}{\sqrt{1+ t}} \cdot \nabla_{\eta} V+ \f{1}{(1+ t)^2} V_{\tau}, 
   \De_y v = \f{1}{(1+ t)^2} \De_{\eta} V,  \\
      L_1 v &=&  \f{1}{{1+ t}} L_1 V, \ \ \ \ 
   H(\phi_{\sigma}, v) = \f{1}{2} \f{1}{(1+ t)^2} D^2 f(\phi_{\f{1}{\sqrt{1+ t} } \Ga}) V^2+E((1+t)^{-1} V),\\ 
    (\nabla_{y} \cdot \sigma)^2 \phi''_{\sigma} &=&  \f{1}{(1+ t)^2} (\nabla_{\eta} \cdot\Ga)^2  
    \phi''_{\f{1}{\sqrt{1+ t}}\Ga }\\ 
   N_2(\sigma, \nabla_y \cdot  \sigma, v)&=&  \f{1}{(1+ t)^2} \left[K_1 ( (1+t)^{-1/2}   \Ga) (\nabla_{\eta}  \cdot \Ga)^2+2K_2 ((1+t)^{-1/2}  \Ga) D^2 f(\phi_{\f{1}{\sqrt{1+ t}}\Ga}) \dpr{\psi}{V^2}\right] \\  
   &+&  K_2 ((1+t)^{-1/2} \Ga)  \dpr{\psi}{E((1+t)^{-1} V)} + \\
   &+& \f{1}{{1+ t}} K_2( (1+t)^{-1/2} \Ga) \langle \psi, (D f(\phi_{\f{1}{\sqrt{1+ t}}\Ga})- D f(\phi) ) V\rangle  =: \f{1}{(1+ t)^2} N_2(\Ga, \nabla_{\eta} \cdot \Ga, V) \\ 
   \sigma_t&=&- \f{1}{2} \f{1}{(1+ t)^{\f{3}{2}}} \Ga- \f{1}{2} \f{1}{(1+ t)^{\f{3}{2}}} \f{y}{\sqrt{1+ t}} \cdot \nabla_{\eta} \cdot \Ga+ \f{1}{(1+ t)^{\f{3}{2}}} \Ga_{\tau}, \ \   \De_{y} \sigma = \f{1}{(1+ t)^{\f{3}{2}}} \De_{\eta} \Ga \\
   N_1(\sigma, \nabla_y \cdot \sigma, v)&=& \f{1}{(1+ t)^2} N_2(\Ga, \nabla_{\eta} \cdot \Ga, V)\phi'_{\f{1}{\sqrt{1+ t}}\Ga } + \f{1}{{1+ t}} \big( D f(\phi_{\f{1}{(\sqrt{1+ t})}\Ga})- D f(\phi) \big) V+ \\
   &+& \f{1}{(1+ t)^2} (\nabla_{\eta}\cdot \Ga)^2 \phi''_{\f{1}{\sqrt{1+ t}}\Ga } =:  \f{1}{(1+ t)^2} N_1(\Ga, \nabla_{\eta} \cdot \Ga, V). 
   \end{eqnarray*}
  So, we have introduced a new set of nonlinearities, which in the new variables $(\tau, \eta)$ take the form  
  \begin{eqnarray*}  
  H(\Ga, V)&=& \f{1}{2}  D^2 f(\phi_{e^{-\f{\tau}{2}} \Ga}) V^2+ e^{2\tau} E(e^{-\tau} V),\\ 
  N_2(\Ga, \nabla_{\eta} \cdot  \Ga, V)&=&   K_1 (e^{- \f{\tau}{2}} \Ga) (\nabla_{\eta} \cdot \Ga)^2+ \f{1}{2 }  K_2(e^{- \f{\tau}{2}} \Ga) \bigg(    D^2 f(\phi_{e^{- \f{\tau}{2}} \Ga}) \langle V^2, \psi\rangle \\ \nonumber
  &+& e^{2\tau} K_2(e^{-\f{\tau}{2}}\Ga)\dpr{\psi}{E(e^{-\tau} V)} + 2 e^{\tau}  \langle \psi, (D f(\phi_{e^{- \f{\tau}{2}}\Ga})- D f(\phi) ) V\rangle \bigg), \\ 
  N_1(\Ga, \nabla_{\eta} \cdot  \Ga, V)&=&  N_2(\Ga, \nabla_{\eta}\cdot  \Ga, V)\phi'_{e^{-\f{\tau}{2}} \Ga } + e^{\tau} \big( D f(\phi_{e^{- \f{\tau}{2}}\Ga})- D f(\phi) \big) V+ 
  e^{- \f{\tau}{2}} (\nabla_{\eta} \cdot \Ga)^2 \phi''_{e^{-\f{\tau}{2}} \Ga }.
  \end{eqnarray*}
  Therefore the system \eqref{vv} is  transfered into the   system 
  \begin{eqnarray}\label{VV}
  \left\{ 
  \begin{array}{l}
  V_{\tau}= (\mathcal{L}_{\eta}+ \f{1}{2}) V+ e^{\tau} L_1 V+ Q_0 H(\Ga, V)+ Q_0 N_1 (\Ga, \nabla_{\eta}   \cdot \Ga, V)\\
  \Ga_{\tau}= \mathcal{L}_{\eta} \Ga+  e^{-\f{\tau}{2}}  N_2 (\Ga, \nabla_{\eta} \cdot  \Ga, V)
  \end{array}  
  \right.
  \end{eqnarray}
  where $H, N_1, N_2$ are defined above  and the operator $\mathcal{L}_{\eta}$ is defined as 
  \begin{equation}
  \label{L_eta}
  \mathcal{L_{\eta}}= \De_{\eta}+ \f{1}{2} \eta \cdot \nabla_{\eta}+ \f{1}{2}.
  \end{equation}
  We finish this subsection by stating the variation of constant formula for \eqref{VV}. Note that this is slightly non-standard, due to the $\tau$ dependence of the linear operator, i.e. the term $e^\tau L_1$ term, in the equation for $V$. It should be noted that  $L_1$ generates a $C_0$ semigroup on the Sobolev space $H^1(\rone)$ (see Lemma \ref{le:l1} below), while  the operator $L_\eta$ generates a semigroup, but on specific weighted $L^2$ based spaces, see Section \ref{sec:2.2} below. 
  Thus, since the action in the variable $z$ and the variable $\eta$ are independent, we may in fact write the system for $(V, \Ga)$ as follows 
   \begin{eqnarray}  
   \label{110} 
   V &=&  e^{\tau (\mathcal{L}_{\eta}+ \f{1}{2})} e^{e^{\tau} L_1} V_0+  \int_0^{\tau} e^{(\tau -s) (\mathcal{L}_{\eta}+ \f{1}{2})} e^{(e^{\tau}- e^s) L_1}[Q_0 H(\Ga, V)+ Q_0 N_1 (\Ga, \nabla_{\eta}  \cdot \Ga, V) (s)]ds\\ 
   \label{120} 
   \Ga &=&  e^{\tau \mathcal{L_{\eta}}}  \Ga_0+  \int_0^{\tau} e^{(\tau -s) \mathcal{L_{\eta}}}    
   e^{-\f{s}{2}}  N_2 (\Ga, \nabla_{\eta} \cdot  \Ga, V) (s) ds,
   \end{eqnarray}
  where $V_0, \Ga_0$ are the initial data of the variables $V, \Ga$. Note that by the scaling variables assignments, $V_0(z,y)=v_0(z,y), \Ga_0(y)=\si_0(y)$. 
  
  It becomes clear by this last formulas that in order to study the long time properties of the system \eqref{110}, \eqref{120}, it will be helpful to know about spectral properties of $L_\eta$ and estimates of the associated  semigroup.

 \subsection{The operator $L_{\eta}$ - spectral information and  the associated semigroup} 
 \label{sec:2.2} 
 For this section, note that the spaces that we introduce are based on $\rne$, instead of the usual $\rn$. This is due to the fact that the scaling variables transformation is performed only in the variables $y\in\rne$. 
 
 The following results are  due to Gallay-Wayne, see Theorem A.1 in \cite{GW}. Note however that the operator  $\cl$  appearing in \cite{GW}, satisfies   $\cl_\eta=\cl - \f{N-1}{2}$ and $N=n-1$. 
 \begin{proposition}
 	\label{prp:01} 
 	Let $m\geq 0$ and $\mathcal{L}_{\eta}$ be the linear operator \eqref{L_eta} acting on $L^2(m)$, and \\  $G(\eta)=  (4\pi)^{-\f{n-1}{2}} e^{- \f{|\eta|^2}{4}}$.  Then, its spectrum consists of\footnote{this is a not necessarily disjoint partition, as some eigenvalues are embedded into the continuous spectrum} $\si(\cl_\eta)=\sigma_d(\mathcal{L}_{\eta})\cup \sigma_{c}(\mathcal{L}_\eta)$, where 
 	\begin{enumerate}
 		\item   The discrete spectrum is 
 		$$
 		\sigma_d(\mathcal{L}_{\eta})=  \bigg\{\la_k \in \mathbb{C}: \la_k= -\f{n+ k- 2}{2};  k= 0, 1, 2, \cdots \bigg\}.
 		$$
 		\item  The continuous spectrum is  
 		$$
 		\sigma_{c}(\mathcal{L}_\eta)=  \bigg\{ \la \in \mathbb{C}:  \Re \la  \leq  -\f{n+5}{4} - \f{m}{2}  \bigg \}.
 		$$
 	\end{enumerate}
 	Moreover, for $m>\f{n-1}{2}$,  the largest element of $\si(\cl_\eta)$ is the eigenvalue $\la_0=-\f{n-2}{2}$ is simple,  with an eigenfunction $G$,  which satisfies 
 	$$
 \cl_\eta G=\la_0 G, \ \ \si(\cl_\eta)\setminus \{-\f{n-2}{2}\}\subset \{\la: \Re\la\leq -\f{n-1}{2}\}
 	$$
 \end{proposition}
 In our next proposition, we discuss the semigroup generation properties. 
 \begin{proposition}
 	\label{prop:semi} 
 	  	The operator $\cl_\eta$ defines a $C_0$ semigroup on $L^2(m)(\mathbb{R}^{n- 1})$.  We have the following  formula for its action 	
 		
 		\begin{eqnarray}
 		\label{semi_1}
 		\widehat{(e^{\tau \cl_{\eta}} f)}(\xi)&=& e^{-\f{n- 2}{2} \tau} e^{- a(\tau) |\xi|^{2}} \widehat{f}(e^{-\f{ \tau}{2}} \xi),  \\ 
 		\label{semi_2}
 		(e^{\tau \cl_{\eta}}f) (\eta)&=& \f{e^{\f{\tau}{2}}}{\big(4 \pi a(\tau) \big)^{\f{n- 1}{2}}} \int_{\rne} G \left(\f{\eta- \eta'}{2 a(\tau)^{\f{1}{2}}}\right) f(e^{\f{\tau}{2}}\eta' ) d \eta',
 		\end{eqnarray} 
 		where $a(\tau)= 1- e^{- \tau}$.  
 \end{proposition}
 The semigroup formulas \eqref{semi_1} and \eqref{semi_2} are also taken from \cite{GW} (see statement $4$, Theorem A.1), with the readjustments due to the different constant and the fact that $L_\eta$ acts on $n-1$ variables.

 Finally, we state some estimates about the action of the semigroup $e^{\tau \cl_\eta}$ on $L^2(m)(\mathbb{R}^{n- 1})$. A version of these are in fact needed for the determination of the spectrum $\si(\cl_\eta)$, but they have already been proved in Proposition A.2, \cite{GW}.  Even though these are well-known, we state them explicitly and provide some calculations for them, as our normalizations are slightly different than \cite{GW}, which may create an element of confusion. 
\subsection{Spectral projections and estimates for $e^{\tau \cl_\eta}$ on $L^2(m)$} 
Fix $m>\f{n}{2}+1$. The spectral projections corresponding to the eigenspaces of $\cl_\eta$ can be constructed explicitly, \cite{GW},  but we will not do so here. Instead, we just construct the one corresponding to the first eigenvalue $\la_0(\cl_\eta)=-\f{n-2}{2}$.  Recall that its eigenspace is one dimensional, spanned by $G$. 
Accordingly, we shall need an eigenvector $e_*$ for the adjoint operator, so that $\cl^* e_*=-\f{n-2}{2} e_*$.  But since 
$$
\cl^{*}_{\eta}= \De_{\eta}- \f{1}{2} \eta \cdot \nabla_{\eta}- \f{n-2}{2}.
$$
So, it is easy to see that $e_*=1$ is an eigenfunction\footnote{belonging to the dual space $L^2(-m)(\rne)$} for $\cl_\eta^*$ and since our normalization for $G$ is chosen so that   $\dpr{1}{G}=(4\pi)^{-\f{n-1}{2}} \int_{\rne} e^{-\f{|\eta|^2}{4}} d\eta=1$, it holds that $e_*=1$.  Thus, we have the convenient formula 
$$
\cp_0 f(\eta)= \left(\int_{\rne} f(\eta') d\eta'\right) G(\eta)=\dpr{f}{1}_\eta G(\eta)
$$
and $\cq_0=Id-\cp_0$. 
 
 \begin{proposition}
 	\label{prop:14}
   Let $m>\f{n+1}{2}$. Then, for all $\al\in {\mathbb N^{n-1}}$,  there exists $C_\al> 0$ such that		\begin{equation}  
 		\label{1_9}
 		\| \nabla^{\al} ( e^{\tau \cl_{\eta}} \mathcal{Q}_0 f)\|_{L^2(m)(\rne)} \leq C_\al  
 		\f{e^{- \f{n- 1}{2}  \tau}}{a(\tau)^{\f{|\al|}{2}}}  \|f\|_{L^2(m)(\rne)},
 		\end{equation}
 		for all $f \in L^2(m)$ and all $\tau> 0$.
 \end{proposition}
 {\bf Remark:} The appearance of the factors $a(\tau)^{\f{|\al|}{2}}$ in the denominator makes the control of second and higher order derivatives, such as  $\nabla^2_\eta \Ga, \nabla^2_\eta V$,  problematic. The reason is that for $0<\tau<1$, $a(\tau)\sim \tau$ and we need an integrable in $\tau$ functions sitting on the right-hand side of \eqref{1_9}. 
 \begin{proof}
 This proposition is proved in \cite{GW}, see Proposition A.2, we have just made the adjustments for the constants and the dimension of the space. Note that the exponent $\f{n-1}{2}$ on the right hand side of the estimate is consistent with the assertion that $\si(\cl_\eta \cq_0)\subset \{\Re\la\leq -\f{n-1}{2}\}$.  
 
 We just copy estimate $(92)$ from Proposition A.2 in \cite{GW},  and we take into account that $\cl_\eta=L-\f{n-2}{2}$, where the operator $L$ is the semigroup generator in \cite{GW}. Thus, we obtain \eqref{1_9}. 
 \end{proof}
  Finally, we need an estimate of the following type. 
 \begin{proposition}
 	\label{prop:16}
 	Let      $m> \f{n}{2}$ and $a\in {\mathbb N}$.  Then, 
 	\begin{equation}
 	\label{20_51} 
 	 \|  \nabla^a  e^{\tau \cl_{\eta}} f\|_{L^{\infty}(\R^{n- 1})} \leq  C \f{e^{- \f{n-2}{2} \tau}}{a(\tau)^{\f{a}{2}}} \bigg( \|f\|_{L^{\infty}(\R^{n- 1})}+ \|f\|_{L^2(m)((\R^{n- 1}))}\bigg).
 	\end{equation}
 	We get the following improvement, when the semigroup is acting on the co-dimension one subspace  $\cq_0[L^2(m)]$ and $m>\f{n}{2}+1$, 
 	\begin{equation}
 	\label{20_55} 
 	\|  \nabla^a  e^{\tau \cl_{\eta}} \cq_0 f\|_{L^{\infty}(\R^{n- 1})} \leq  C \f{e^{- \f{n-1}{2} \tau}}{a(\tau)^{\f{a}{2}}} \bigg( \|f\|_{L^{\infty}(\R^{n- 1})}+ \|f\|_{L^2(m)((\R^{n- 1}))}\bigg).
 	\end{equation}
 	
 \end{proposition}
  \begin{proof}
  	We divide the proof into the cases of $\tau < 1$ and $\tau \geq 1$. 
  	For $\tau< 1$ we use the definition \eqref{semi_2} in our calculations. Indeed,
  	\begin{eqnarray*}
  		&&\| \nabla^a e^{\tau \cl_{\eta}}  f\|_{L^{\infty}} \leq C \f{e^{\f{\tau}{2}}}{(a(\tau))^{\f{n+a- 1}{2}}} \|\int_{\R^{n- 1}}  \nabla^a G(\f{\eta- \eta'}{(a(\tau))^{\f{1}{2}}}) f(e^{\f{\tau}{2}} \eta') d\eta' \|_{L^{\infty}} \\
  		& \leq& C \f{  \| \nabla^a G(\f{ \cdot}{(a(\tau))^{\f{1}{2}}})\|_{L^1(\R^{n- 1})} \| f(e^{\f{\tau}{2}} \cdot)\|_{L^{\infty}(\R^{n- 1})}}{(a(\tau))^{\f{n+ a- 1}{2}}}   \leq \f{C \| \nabla^a  G\|_{L^1(\R^{n- 1})} \| f\|_{L^{\infty}(\R^{n- 1})}}{(a(\tau))^{\f{a}{2}}}  \leq \f{C \| f\|_{L^{\infty}(\R^{n- 1})}}{(a(\tau))^{\f{a}{2}}}.
  	\end{eqnarray*}
  	Since for $\tau< 1$, $e^{ \f{n-2}{2} \tau}$ is bounded, we have 
  	\begin{eqnarray}\label{t_1}
  	&&\| e^{\tau \cl_{\eta}} \nabla^a f\|_{L^{\infty}} \leq \f{C e^{- \f{n-2}{2} \tau} }{(a(\tau))^{\f{a}{2}}} \| f\|_{L^{\infty}(\R^{n- 1})}. 
  	\end{eqnarray}  	
  	We  now turn our attention to the case $\tau \geq 1$. We have, 
  	\begin{eqnarray*}
  		&&\|  \nabla^a e^{\tau \cl_{\eta}}   f\|_{L^{\infty}} \leq C e^{- \f{n- 2}{2} \tau} \| e^{- a(\tau) |\cdot|^2} | |\cdot|^a \widehat{f}(e^{-\f{\tau}{2}} \cdot) \|_{L^1}=  C e^{- \f{n- 2}{2} \tau} \int_{\R^{n-1}}  e^{- a(\tau) |\xi|^2} | \xi|^a |\widehat{f}(e^{-\f{\tau}{2}} \xi)|   d\xi\\
  		&&= e^{- \f{n- 2}{2} 
  			\tau} e^{\f{(n+a-1)\tau}{2}} \int_{\R^{n-1}}  e^{- a(\tau) |e^{\f{\tau}{2}} q|^2} |q|^a |\widehat{f}(q)|  d q\\	
  		&& \leq C e^{\f{a+1}{2}\tau} \bigg[\int_{a(\tau) |e^{\f{\tau}{2}} q|^2 \leq 1}+ \sum_{i= 1}^{\infty} \int_{i \leq a(\tau) |e^{\f{\tau}{2}} q|^2 \leq i+ 1} \bigg] \bigg(e^{- a(\tau) |e^{\f{\tau}{2}} q|^2} |q|^a
  		|\widehat{f}(q)|  \bigg) d q:= J_1+ J_2.
  	\end{eqnarray*}
  	Since $ |\hat{f}(q)|  \leq  \|f\|_{L^1}\leq C\|f\|_{L^2(m)}$, because $m>\f{n}{2}$, we have 
  	\begin{eqnarray*}
  		&&e^{-\f{a+1}{2}\tau} J_1 \leq   \int_{a(\tau) |e^{\f{\tau}{2}} q|^2 \leq 1}  e^{- a(\tau) |e^{\f{\tau}{2}} q|^2} |q|^j |\widehat{f}(q)|   d q \leq   \|f\|_{L^2(m)} \int_{a(\tau)|e^{\f{\tau}{2}} q|^2 \leq 1}  |q|^a d q\\
  		&&\ \leq C  \|f\|_{L^2(m)} \int_{0}^{\f{e^{-\f{\tau}{2}}}{a(\tau)^{\f{1}{2}}}}  r^{a+  n- 2} d r \leq  C \f{e^{-\f{(a+n-1)\tau}{2}} }{a(\tau)^{\f{a+n-1}{2}}} \|f\|_{L^2(m)}\leq C e^{-\f{(a+n-1)\tau}{2}} \|f\|_{L^2(m)},
  	\end{eqnarray*}
  	since for $\tau>1$, $a(\tau)>\f{1}{2}$. 
  	In other words,
  $$
  J_1   \leq  C  e^{-\f{(n-2)}{2}\tau}  \|f\|_{L^2(m)}.
 $$
  	For $J_2$ in a similar way, we have
  	\begin{eqnarray*}
  		&&e^{-\f{a+1}{2}\tau} J_2 \leq  \|f\|_{L^2(m)} \sum_{i= 1}^{\infty} \int_{i \leq a(\tau) |e^{\f{\tau}{2}} q|^2 \leq i+ 1}  e^{- a(\tau) |e^{\f{\tau}{2}} q|^2} |q|^a    d q\\
  		&&
  		\leq  C  \|f\|_{L^2(m)} \sum_{i= 1}^{\infty} e^{- i} \int_{i  \f{e^{- \f{\tau}{2}}}{a(\tau)^{\f{1}{2}}} }^{{(i+ 1)  \f{e^{- \f{\tau}{2}}}{a(\tau)^{\f{1}{2}}} }}  r^{a+ n- 2} d r \leq  C \|f\|_{L^2(m)} 
  		e^{- \f{a+ n-1}{2} \tau}   \sum_{i= 1}^{\infty} e^{- i} \bigg((i+1)^{a+n-1}- i^{a+n-1}\bigg)
  		\\
  		&&\  \leq C \|f\|_{L^2(m)} e^{- \f{a+n-1}{2} \tau}.
  	\end{eqnarray*}
  	In other words, 
  	$$
  	  	J_2   \leq  C   e^{-\f{(n-2)}{2}\tau}  \|f\|_{L^2(m)}. 
  	  	$$
  	Therefore for $\tau> 1$ if we put both estimates for $J_1$ and $J_2$ together we get 
  	\begin{eqnarray}\label{t_2}
  	&&\|\nabla^a  e^{\tau \cl_{\eta}}  f\|_{L^{\infty}} \leq
  	C  e^{- \f{ n - 2}{2} \tau}   \|f\|_{L^2(m)}. 
  	\end{eqnarray}
   The proof of \eqref{20_51} is now is complete by putting  the estimates \eqref{t_1} and \eqref{t_2} together.
   For the estimate \eqref{20_55}, we use that $\cq_0 f=f-\dpr{f}{1}_\eta G$, so that  
   $\dpr{\cq_0 f}{1}_\eta=\dpr{f}{1}_\eta- \dpr{f}{1}_\eta \dpr{G}{1}_\eta=0$. So, 
   $ \widehat{\cq_0 f}(0)=0$. Thus, in the estimates above, we can estimate 
   $$
   | \widehat{\cq_0 f}(q)|=| \widehat{\cq_0 f}(q)-\widehat{\cq_0 f}(0)|\leq |q| 
   \|\nabla \widehat{\cq_0 f}\|_{L^\infty} \leq    C |q| \int_{\rne} |\eta| |\cq_0 f(\eta)| d\eta \leq C |q| 
   \|\cq_0 f\|_{L^2(m)},
   $$
   where in the last inequality, we needed $m>\f{n}{2}+1$. 
   In addition, 
   $$
     \|\cq_0 f\|_{L^2(m)}\leq \|f\|_{L^2(m)}+ |\dpr{f}{1}_\eta| \|G\|_{L^2(m)} \leq C  \|f\|_{L^2(m)}. 
   $$ 
   Plugging these estimates in the argument above, we gain a power of $|q|$, which gains   an extra power of $e^{-\f{\tau}{2}}$ over the estimate \eqref{20_51}, which is reflected on the right-hand side of \eqref{20_55}.    
  \end{proof}

  \section{Long time asymptotics - setup and further reductions}
  \label{sec:3} 
  In this section, we study the precise asympotics of the radiation term $V$ and the phase $\Ga$. 
  \subsection{Decomposing the evolution along the spectrum of $\cl_\eta$} Due to the fairly explicit  spectral information available about $\cl_\eta$, see Proposition \ref{prp:01}, and the semigroup estimates in Propositions \ref{prop:14} and \ref{prop:16}, it is beneficial to consider the system   \eqref{110}, \eqref{120} in $L^2(m)$ based spaces. For the estimates to work, we need to take $m$ to be large enough, say $m> \f{n+1}{2}$.   In this space, the operator $\cl_{\eta}$ has at least one isolated eigenvalue $\la_0= - \f{n- 2}{2}$ corresponding  to the eigenfunction 
  $G(\eta)=  (4\pi)^{-\f{n-1}{2}}  e^{- \f{|\eta|^2}{4}}$, recall $\eta \in \mathbb{R}^{n-1}$.  
  
   For conciseness, we set $\widetilde{f}= \mathcal{Q}_0 f$, that is all functions with a tilde hereafter will denote functions in $\cq_0(L^2(m))$.   With this set up, we decompose the solutions of the system of equations \eqref{VV}  in the following way,
  
  \begin{eqnarray} \label{decomposit}
  \left\{ 
  \begin{array}{l}
  V(z, \eta, \tau)=\al(z, \tau) G(\eta)+ \widetilde{V}(z, \eta, \tau),\\
  \Ga(\eta, \tau)=\ga(\tau) G(\eta)+ \widetilde{\Ga}(\eta, \tau),
  \end{array}  
  \right.
  \end{eqnarray}
  where  $\al(z, \tau)= \dpr{V}{1}_\eta=\int_{\rne} V(z,\eta, \tau) d\eta$  and $\ga(\tau)= \dpr{\Ga}{1}_\eta=  \int_{\rne} \Ga(\eta, \tau) d\eta$.  
  In order to  find the  representations of $\al$ and $\ga$ we make $\dpr{\cdot}{1}$ in \eqref{VV}, 
  \begin{eqnarray*} 
  	\left\{ 
  	\begin{array}{l}
  		\al_{\tau} = \langle V_{\tau}, 1\rangle_{\eta}= \langle (\mathcal{L}_{\eta}+ \f{1}{2}) V, 1 \rangle_{\eta}+ e^{\tau} \langle L_1 V, 1 \rangle_{\eta}+ \langle Q_0 H(\Ga, V), 1 \rangle_{\eta}+ \langle Q_0 N_1 (\Ga, \nabla_{\eta} \Ga, V), 1 \rangle_{\eta}\\
  		\ga_{\tau}= \langle \Ga_{\tau}, 1 \rangle_{\eta}= \langle \mathcal{L}_{\eta} \Ga, 1 \rangle_{\eta}+  e^{-\f{\tau}{2}}  \langle N_2 (\Ga, \nabla_{\eta} \Ga, V),  1 \rangle_{\eta}.
  	\end{array}  
  	\right.
  \end{eqnarray*}
  Some of the terms in this system can be simplified as follows 
  \begin{eqnarray*} 
  	&&\langle (\mathcal{L}_{\eta}+ \f{1}{2}) V, 1 \rangle_{\eta}=  \langle \De V, 1 \rangle_{\eta}+ \f{1}{2} \langle \eta \cdot \nabla_{\eta} V, 1 \rangle_{\eta}+  \langle V, 1 \rangle_{\eta}= \f{1}{2} \int \eta \cdot \nabla V\ d\eta+ \langle V, 1 \rangle_{\eta}= - \f{n- 3}{2} \al(z, \tau),\\
  	&& \langle L_1 V, 1 \rangle_{\eta}= L_1 \al(z, \tau). 
  \end{eqnarray*}
  Therefore, we obtain the ODE/PDE system 
   \begin{equation}   
   \label{140} 
   \left\{ 
   \begin{array}{l}
   \al_{\tau}(z, \tau) = - \f{n- 3}{2}\ \al(z, \tau) + e^{\tau}  L_1 \al(z, \tau)+ \langle Q_0 H(\Ga, V), 1 \rangle_{\eta}+ \langle Q_0 N_1 (\Ga, \nabla_{\eta} \cdot   \Ga, V), 1 \rangle_{\eta}\\
   \ga_{\tau}= - \f{n- 2}{2}\ \ga(\tau)+  e^{-\f{\tau}{2}}  \langle N_2 (\Ga, \nabla_{\eta}\cdot  \Ga, V),  1 \rangle_{\eta}.
   \end{array}  
   \right.
   \end{equation}
  Recall now that by our construction in \eqref{vv}, we had $v=Q_0 v$ or equivalently $P_0 v=0$. Clearly, such a property transfers to the scaling variables\footnote{the operators $P_0, Q_0$ are  acting in the variable $z$, which is independent on the action in the scaled variable $\eta$}, that is $Q_0 V=V, P_0 V=0$. Consequently, 
  $$
  P_0 \al(\cdot, \tau)=P_0 \dpr{V(\cdot, \eta, \tau)}{1}_\eta = \dpr{P_0  V(\cdot, \eta, \tau)}{1}_\eta=0
  $$
  or equivalently $\al(z, \tau)=Q_0 \al(\cdot, \tau)$.   
  Thus, the system \eqref{140}, which consists of an ODE and a PDE, has the following integral representation, 
   \begin{eqnarray} 
   \label{200} 
   		\al(z, \tau) &=&  e^{- \f{n- 3}{2} \tau} e^{e^{\tau}  L_1} Q_0 \al(z,0) +  \\
   		\nonumber 
   		&+& \int_0^{\tau} e^{- \f{n- 3}{2} (\tau- s)} e^{(e^{\tau}- e^{s})  L_1} Q_0 \Bigg[ \langle  H(\Ga, V), 1 \rangle_{\eta}(s) + \langle   N_1 (\Ga, \nabla_{\eta} \cdot \Ga, V), 1 \rangle_{\eta}(s) \Bigg] ds,\\
   		\label{205} 
   		\ga(\tau) &=&  e^{- \f{n- 2}{2} \tau} \ga(0)+   \int_0^{\tau} e^{- \f{n- 2}{2}(\tau- s)} e^{-\f{s}{2}} \langle N_2 (\Ga, \nabla_{\eta} \cdot \Ga, V),  1 \rangle_{\eta}(s) ds.
   \end{eqnarray}
  We also can find the representation of $\widetilde{V}$ and $\widetilde{\Ga}$. For that, 	we  project the system of equations \eqref{VV} away from the eigenvector $G$. 
  That is, we apply $\cq_0$ in \eqref{VV}. Note that all operations in the $z$ variable commute with the operations in the $\eta$ variables, such as  $L_1 \cq_0=\cq_0 L_1, \cq_0 Q_0=Q_0 \cq_0$ and so on. We obtain  
  \begin{eqnarray*}
  		\widetilde{V}_{\tau} &=&   (\mathcal{L}_{\eta}+ \f{1}{2}) \widetilde{V}+ 
  		e^{\tau} L_1 \tilde{V} + Q_0[\mathcal{Q}_0 H(\Ga, V)+ \mathcal{Q}_0   N_1 (\Ga, \nabla_{\eta}\cdot  \Ga, V)], \\
  		\widetilde{\Ga}_{\tau} &=&  \mathcal{L}_{\eta} \widetilde{\Ga}+  e^{-\f{\tau}{2}}  \mathcal{Q}_0 N_2 (\Ga, \nabla_{\eta} \cdot \Ga, V). 
  \end{eqnarray*}
 Note that once again $\widetilde{V}(z, \eta, \tau)= Q_0 \widetilde{V}(z, \eta, \tau)$. The  system has the following integral representation,
  \begin{eqnarray} 
  \label{210}
  	  		\widetilde{V}(z, \eta, \tau) &=&  e^{(\cl_{\eta}+ \f{1}{2}) \tau}   e^{e^{\tau} L_1} Q_0 \widetilde{V}_0+  \\
  	  		\nonumber 
  	  		&+& \int_0^{\tau} e^{(\cl_{\eta}+ \f{1}{2}) (\tau- s)} \mathcal{Q}_0 e^{(e^{\tau}- e^s) L_1} Q_0 
  	  		\Big[  H(\Ga, V)(s)+   N_1 (\Ga, \nabla_{\eta}\cdot   \Ga, V)(s)\Big] ds\\
  	  		\label{215}
  		\widetilde{\Ga}(\eta, \tau) &=&  e^{\tau \cl_{\eta} }  \widetilde{\Ga}_0+  \int_0^{\tau} e^{\cl_{\eta} (\tau- s)}  \mathcal{Q}_0  e^{-\f{s}{2}}  N_2 (\Ga, \nabla_{\eta}\cdot   \Ga, V) (s) ds.
  \end{eqnarray}
  Thus, we have reduced matters to the system \eqref{200}, \eqref{205}, \eqref{210}, \eqref{215}. Our next goal is to show a small data, global regularity result for this system. 
  \subsection{The function space}
  We now introduce a function space $X$. Of course,  the time decay exponents are chosen appropriately so that the argument eventually closes. More specifically, 
  \begin{eqnarray*}
  	& &\|(\al, \be, \widetilde{V}, \widetilde{\Ga})\|_{X} := 
  	\sup_{\tau > 0} \left\{ e^{(n-\f{1}{2}) \tau} \|\al(\cdot, \tau)\|_{H^1_z}+ 
  	e^{\f{n- 2}{2} \tau} |\ga(\tau)|\right\}+ \\
  	&+& 	\sup_{\tau > 0} \left\{ e^{ (n-\f{1}{2}) \tau} 
  	\| \widetilde{V}\|_{L^2(m)H^1_z} +   	e^{(n-\f{1}{2}) \tau}   \| \widetilde{V}\|_{L^\infty_\eta H^1_z }   \right\} + \\
  	&+& 
  		\sup_{\tau > 0} \left\{ e^{\f{n- 1}{2} \tau} \| \widetilde{\Ga}\|_{H^1(m)}+   e^{\f{n- 1}{2} \tau} \| \widetilde{\Ga}\|_{L^{\infty}_{\eta}}+ e^{\f{n- 1}{2} \tau} \|\nabla_{\eta} \widetilde{\Ga}\|_{L^{\infty}_{\eta}}  \right\}. 
  \end{eqnarray*}
  Here, recall  the convention   $\|f\|_{L^\infty_\eta H^1_z}=\| \|f \|_{H^1_z} \|_{L^\infty_\eta}$. 
  \subsection{Asymptotics in the scaling variables system}
  The following is the main result, describing the asymptotics of the evolution in the scaling variables. 
  We just note that by the setup in the scaling variables, the initial data in the scaling variables coincides with the initial data in the original variables. 
  \begin{theorem}
  	\label{theo:20} 
  	There exists $\eps_0>0$ and a constant $C_0$, so that for every $\eps: 0<\eps<\eps_0$ and initial data $(\al_0, \ga_0, \widetilde{V}_0, \widetilde{\Ga}_0)=(\al, \ga, \widetilde{V}, \widetilde{\Ga})|_{\tau=0}$ satisfying 
  	\begin{equation}
  	\label{305}
    \|\al(\cdot, 0)\|_{H^1_z}+ 
  	  |\ga(0)| + 
  		\| \widetilde{V}_0\|_{H^1_z H^1(m)} +     \| \widetilde{V}_0\|_{H^1_z L^\infty_\eta}     +  
  		 \| \widetilde{\Ga}_0\|_{H^1(m)}+     \| \widetilde{\Ga}_0\|_{L^{\infty}_{\eta}}+   \|\nabla_{\eta} \widetilde{\Ga}_0\|_{L^{\infty}_{\eta}}   <\eps, 
  	\end{equation}
  	the system \eqref{200}, \eqref{205}, \eqref{210}, \eqref{215} has an unique solution in the ball $B_X(0, C_0 \eps)$, with the given initial data. That is,  it satisfies 
  	\begin{eqnarray}
  	\label{700} 
  	\|\al(\cdot, \tau)\|_{H^1_z}\leq C_0\eps e^{-(n-\f{1}{2}) \tau},  |\ga(\tau)|\leq C_0 \eps e^{-\f{n- 2}{2} \tau} \\
  	\label{710} 
  	\| \widetilde{V}(\tau, \cdot)\|_{L^2(m)H^1_z\cap L^\infty_\eta H^1_z } \leq C_0 \eps  e^{ -(n-\f{1}{2})  \tau} \\
  	\label{720} 
  	\| \widetilde{\Ga}(\tau, \cdot)\|_{H^1(m)\cap L^\infty_\eta}+
  	\|\nabla_{\eta} \widetilde{\Ga}(\tau, \cdot)\|_{L^{\infty}_{\eta}} \leq C_0\eps e^{-\f{n- 1}{2} \tau}. 
  	\end{eqnarray}
  	In particular, taking into account \eqref{decomposit}, 
  	\begin{eqnarray}
  	\label{718} 
  		& & 	\| V(\tau, \cdot)\|_{L^2(m)H^1_z\cap L^\infty_\eta H^1_z } \leq C_0 \eps  e^{ -(n-\f{1}{2})  \tau} \\
  		\label{719}  	& &  	\| \Ga(\tau, \cdot)\|_{H^1(m)\cap L^\infty_\eta}+
  		\|\nabla_{\eta} \Ga (\tau, \cdot)\|_{L^{\infty}_{\eta}} \leq C_0\eps e^{-\f{n- 2}{2} \tau}. 
  	\end{eqnarray}
  \end{theorem}
  The proof of Theorem \ref{theo:20} occupies Section \ref{sec:4} below. 
   We only mention that as a consequence of it and  the relations \eqref{718}, \eqref{719}, we derive the asymptotics of the solutions $(v, \si)$ of the system \eqref{vv}.  More precisely, taking into account the scaling variables definition, we obtain 
   \begin{eqnarray*}
    & &  \|\si(t, \cdot)\|_{L^\infty_y}\leq C \eps_0 (1+t)^{-\f{n-1}{2}}, \ \     
    \|\nabla_y \si(t,\cdot)\|_{L^\infty_y}\leq C \eps_0 (1+t)^{-\f{n}{2}}\\ 
    & & \|v\|_{L^\infty_{\eta, z}}\leq C\eps_0 (1+t)^{-(n+\f{1}{2})} 
   \end{eqnarray*}
   These are precisely the claims in  \eqref{1.10}, \eqref{1.11} and \eqref{1.12}.

  \section{Long time asymptotics - Proof of Theorem \ref{theo:20}}
  \label{sec:4} 
  We perform a fixed point argument in a sufficiently small ball of $X$. To that end, we view the question for solvability  as a fixed point problem in  the schematic  form 
  $$
  (\al, \ga, \widetilde{V}, \widetilde{\Ga})=\textup{free solutions}+ \Phi(\al, \ga, \widetilde{V}, \widetilde{\Ga}), 
  $$
  where $\Phi$ is defined as the Duhamel terms in the right-hand sides of \eqref{200}, \eqref{205}, \eqref{210}, \eqref{215}.  The existence and uniqueness of the fixed point will be established, once we can show that there exists a sufficiently small $\eps>0$ and a $C$ (depending on parameters, but not on $\eps$), so that whenever initial data satisfies \eqref{305}, we have 
  \begin{itemize}
  	\item  \begin{equation}
  	\label{300} 
  	\|\textup{free solutions}\|_X\leq C \eps, 
  	\end{equation}
  	\item For all  $(\al, \ga, \widetilde{V}, \widetilde{\Ga})\in X: \|(\al, \ga, \widetilde{V}, \widetilde{\Ga})\|_{X}\leq \eps$, there is 
  	\begin{equation}
  	\label{310} 
  	\|\Phi(\al, \ga, \widetilde{V}, \widetilde{\Ga})\|_X\leq C \eps^2.
  	\end{equation}
  	\item For all $(\al_j, \ga_j, \widetilde{V}_j, \widetilde{\Ga}_j): 
  	\|(\al_j, \ga_j, \widetilde{V}_j, \widetilde{\Ga}_j)\|_{X}\leq \eps, j=1, 2$, there is 
  		\begin{equation}
  		\label{311} 
  		\|\Phi(\al_1, \ga_1, \widetilde{V}_1, \widetilde{\Ga}_1)-\Phi(\al_2, \ga_2, \widetilde{V}_2, \widetilde{\Ga}_2)\|_X\leq C \eps\|(\al_1, \ga_1, \widetilde{V}_1, \widetilde{\Ga}_1)- (\al_2, \ga_2, \widetilde{V}_2, \widetilde{\Ga}_2)\|_X. 
  		\end{equation}
  \end{itemize}
  Due to the multilinear structure of the functional $\Phi$, we can concentrate on \eqref{310}, identical approach will yield \eqref{311}. 
  We start with the free solutions, as these only involve the mapping properties of the semigroups $e^{\tau \cl_\eta}$ and $e^{s L_1}$. 
  \begin{lemma}
  	\label{le:l1}
  	The operator $L_1$ generates a semigroup on $H^1(\rone)$. In fact, under the Assumption \ref{assumption}, 
  	for all $\de_1<\de$, there is a constant $C=C_{\de_1}$, 
  	\begin{equation}
  	\label{20_5}
  	\|e^{s L_1} Q_0 f\|_{H^1(\rone)}\leq C_{\de_1}    e^{- \de_1 s} \|f\|_{H^1(\rone)}.
  	\end{equation}	
  	In the applications, we will use $\de_1:=\f{\de}{2}$. 
  \end{lemma}
  The proof of Lemma \ref{le:l1} involves the spectral gap property assumption. It is done by combining appropriate resolvent estimates and the Gearheart-Pr\"uss theorem, it   is postponed to the Appendix. 
  
   Using the positivity properties of the function $G$, we have the following 
   \begin{lemma}
   	\label{ex}
   	Let $1 \leq p \leq \infty$, then  there is the pointwise inequality 
   	\begin{equation}
   	\label{303} 
   	\| e^{\tau \cl_{\eta}} f(\cdot, \eta) \|_{L^p_z(\R)} \leq e^{\tau \cl_{\eta}}	\|  f(\cdot, \eta) \|_{L^p_z(\R)}
   	\end{equation}
   \end{lemma}
   \begin{proof}
   	Based on the semigroup definition of \eqref{semi_2}, and considering the fact that $G(\cdot)$ is a positive function of the variable $\eta$,
   	\begin{eqnarray*}
   		\| e^{\tau \cl_{\eta}} f(\cdot, \eta) \|_{L^p_z(\R)}&=& \f{e^{\f{\tau}{2}}}{(4 \pi a(\tau))^{\f{n- 1}{2}}} \| \int_{\R^{n- 1}} G(\f{\eta- \eta'}{2(a(\tau))^{\f{1}{2}}}) f(\cdot,e^{\f{\tau}{2}} \eta') d\eta'\|_{L_z^p(\R)}\leq \\
   		&\leq & \f{e^{\f{\tau}{2}}}{(4 \pi a(\tau))^{\f{n- 1}{2}}}  \int_{\R^{n- 1}} G(\f{\eta- \eta'}{2(a(\tau))^{\f{1}{2}}}) \|f(\cdot, e^{\f{\tau}{2}} \eta')\|_{L_z^p(\R)} d\eta'= e^{\tau \cl_{\eta}}	\|  f(\cdot, \eta) \|_{L^p_z(\R)}
   	\end{eqnarray*}
   \end{proof}
  
  \subsection{Control of the free solutions} 
  \label{sec:4.1} 
  For the free solution term of $\al$, we have by \eqref{20_5}, with 
  $$
  e^{- \f{n- 3}{2} \tau}   \|e^{e^{\tau}  L_1} Q_0 \al(z,0)\|_{H^1_z}\leq C e^{- \f{n- 3}{2} \tau} e^{-\f{\de}{2} e^\tau} \|\al(z,0)\|_{H^1_z}\leq C\eps e^{-(n-\f{1}{2})\tau},
  $$
  where we gave up an exponential decay in $e^\tau$. 
  For the free solution term of $\ga$, we clearly have $e^{-\f{n-2}{2}\tau} |\ga(0)|\leq \eps e^{-\f{n-2}{2}\tau}$. 
  
  For the free solution of $\tilde{V}$, we need to control two terms. We have by \eqref{1_9} and \eqref{20_5} 
  \begin{eqnarray*}
  \|e^{(\cl_{\eta}+ \f{1}{2}) \tau}  e^{e^{\tau} L_1} Q_0 \widetilde{V}_0\|_{L^2(m) H^1_z} \leq 
  C e^{-\f{\de}{2}e^\tau} e^{-\f{n-2}{2}\tau} \|  \widetilde{V}_0\|_{L^2(m) H^1_z}\leq C\eps  e^{-(n-\f{1}{2}) \tau},
  \end{eqnarray*}
  where we gave up an exponential decay in $e^\tau$ as well. 
  For the other free solution term of  $\tilde{V}$, we have by   \eqref{303}, \eqref{20_5}  and \eqref{1_9} 
   \begin{eqnarray*}
   	\|e^{(\cl_{\eta}+ \f{1}{2}) \tau}   e^{e^{\tau} L_1} Q_0 \widetilde{V}_0\|_{L^\infty_\eta H^1_z} & \leq &  C 	\|e^{(\cl_{\eta}+ \f{1}{2}) \tau}  \| e^{e^{\tau} L_1} Q_0 \widetilde{V}_0\|_{H^1_z} 
   	\|_{L^\infty_\eta}\leq \\
   	&\leq & C e^{-\f{(n-2)}{2}\tau} e^{-\f{\de}{2}e^\tau } (\|\widetilde{V}_0\|_{L^\infty_\eta H^1_z}+
   	\|\widetilde{V}_0\|_{L^2(m) H^1_z})\leq C\eps  e^{-(n-\f{1}{2}) \tau}. 
   \end{eqnarray*}  
   For the free solution of the $\tilde{\Ga}$, we have by \eqref{1_9} and \eqref{20_55}, 
   $$
   \| e^{\tau \cl_\eta} \tilde{\Ga}_0\|_{L^\infty_\eta\cap L^2(m)}  \leq  C e^{-\f{n-1}{2}\tau} \|\tilde{\Ga}_0\|_{L^\infty_\eta\cap L^2(m)}. 
   $$
   For the terms $\|\nabla_\eta e^{\tau \cl_\eta} \tilde{\Ga}_0\|_{L^\infty_\eta\cap L^2(m)}$, we split our considerations in two cases, $\tau<1, \tau\geq 1$. 
 We consider the case $\tau<1$ first.  By  a formula equivalent to   \eqref{semi_2} 
 \begin{eqnarray*}
   \|\nabla_\eta e^{\tau \cl_\eta} \tilde{\Ga}_0\|_{L^\infty_\eta\cap L^2(m)}  &\leq &   \f{C}{\big(a(\tau) \big)^{\f{n- 1}{2}}} \|\int_{\rne} G \left(\f{\eta'}{2 a(\tau)^{\f{1}{2}}}\right) \nabla_\eta \tilde{\Ga}_0 (e^{\f{\tau}{2}}(\eta-\eta') ) d \eta'\|_{L^\infty_\eta\cap L^2(m) } \leq  \\
   &\leq & C\|\nabla_\eta 
   \tilde{\Ga}_0\|_{L^\infty_\eta\cap L^2(m)}\leq C\eps e^{-\f{n-1}{2}\tau}. \\
 \end{eqnarray*}  
 since $e^{\f{n-1}{2}\tau} $ is bounded for $0<\tau\leq 1$.       Finally for $\tau>1$, we have   that $a(\tau)\geq \f{1}{2}$, so we conclude from \eqref{20_55} 
   $$
   \|\nabla_\eta e^{\tau \cl_\eta} \tilde{\Ga}_0\|_{L^\infty_\eta}\leq  C  e^{-\f{n-1}{2}\tau}  \|\tilde{\Ga}_0\|_{L^2(m)} \leq C\eps e^{-\f{n-1}{2}\tau}
   $$ 
  This completes the cases of the free solutions. 
  
  Below, we shall use the semigroup estimates on the Duhamel terms in the same way we have used them on the free solutions. This will bring about certain norms on the nonlinear terms, so we  need to prepare these estimates. 
  \subsection{Estimates on the nonlinear terms $H(\Ga, V)$, $N_1 (\Ga, \nabla_{\eta} \Ga, V)$ and $N_2(\Ga, \nabla_{\eta} \Ga, V)$}
  We first note that due to \eqref{decomposit}, we have the following estimates 
  \begin{eqnarray*}
  & & \|V\|_{L^2(m) H^1_z}+\|V\|_{L^\infty  H^1_z}\leq \|\al(s, \cdot)\|_{H^1_z} (\|G\|_{L^\infty_\eta}+ \|G\|_{L^2(m)})+  \|\tilde{V}(s, \cdot)\|_{L^2(m) H^1_z}+\|\tilde{V}(s, \cdot)\|_{L^\infty_\eta  H^1_z},\\
  & & \|\Ga\|_{H^1(m)}+ \|\Ga\|_{W^{1, \infty}_\eta}  \leq |\ga(s)| 
  (\|G\|_{H^1(m)}+ \|G\|_{W^{1, \infty}_\eta}) + \|\tilde{\Ga}(s, \cdot)\|_{H^1(m)}+ 
  	\|\tilde{\Ga}(s, \cdot)\|_{W^{1, \infty}_\eta}.
  \end{eqnarray*}
   Thus, if $(\al, \ga, \widetilde{V}, \widetilde{\Ga})\in X: \|(\al, \ga, \widetilde{V}, \widetilde{\Ga})\|_X<\eps$, we conclude that the corresponding $(V, \Ga)$, given by \eqref{decomposit} satisfy 
   \begin{eqnarray}
   \label{500} 
  & &  \|V(s, \cdot)\|_{L^2(m) H^1_z}+\|V(s, \cdot)\|_{L^\infty_\eta  H^1_z}\leq C\eps e^{-(n-\f{1}{2}) s}, \\
  \label{510} 
  & &  \|\Ga(s, \cdot)\|_{H^1(m)}+ \|\Ga(s, \cdot)\|_{W^{1, \infty}_\eta}  \leq  C\eps e^{-\f{n-2}{2} s}, 
   \end{eqnarray}
   With that in mind, we present the following lemma. 
   \begin{lemma}
   	\label{nonl} 
   	Let $(V, \Ga)$ be as in \eqref{decomposit} and $(\al, \ga, \widetilde{V}, \widetilde{\Ga})\in X: \|(\al, \ga, \widetilde{V}, \widetilde{\Ga})\|_X<\eps$. Then, the nonlinearities $H(\Ga, V)$, $N_1 (\Ga, \nabla_{\eta}\cdot  \Ga, V)$ and $N_2(\Ga, \nabla_{\eta}\cdot  \Ga, V)$ obey the following bounds
   	\begin{eqnarray}
   	\label{520} 
   	\|H(\Ga, V)(s)\|_{L^2_\eta(m) H^1_z }\leq C\eps^2 e^{- (2 n-1) s}.  \\
   	\label{530} 
   		\|N_2 (\Ga, \nabla_{\eta}\cdot \Ga, V)\|_{L^2(m)}+
   		\| N_2(\Ga, \nabla_{\eta} \cdot\Ga, V)\|_{L^\infty_\eta}\leq C \eps^2 e^{-(n-2)s},  \\ 
   		\label{540} 
   	\|Q_0 N_1 (\Ga, \nabla_{\eta}\cdot \Ga, V)\|_{L^2(m) H^1_z }+
   	\|Q_0 N_1 (\Ga, \nabla_{\eta}\cdot \Ga, V)\|_{L^\infty_\eta H^1_z}\leq C\eps^2 e^{-(n-\f{3}{2})s}  
   	\end{eqnarray}
   \end{lemma}
 {\bf Remark:} Note that the spectral projections $Q_0, \cq_0$ appear in front of all nonlinearities displayed above. In almost all cases, that is for \eqref{520} and \eqref{530}, this does not make a difference in the bounds (i.e. the exponents on the right-hand side). The appearance of $Q_0$ in \eqref{540} though makes a difference (and even then, for only one term). Nevertheless, the estimate \eqref{540} without $Q_0$ holds with the weaker exponent $e^{-(n-2)s}$ on the right-hand side.
   \begin{proof}
   	Note that by Sobolev embedding, we have the {\it a priori} bound on $\|V\|_{L^\infty}$ as follows 
   \begin{equation}
   \label{550}
   	\|V(s)\|_{L^\infty_{z,\eta}} \leq C \|V(s, \cdot)\|_{L^\infty_\eta  H^1_z}\leq C\eps e^{-(n-\f{1}{2})  s}.
   \end{equation}
 We start with the estimate for $H(\Ga, V)=\f{1}{2}  D^2 f(\phi_{e^{-\f{s}{2}} \Ga}) V^2+ e^{2s} E(e^{-s} V)$.  We have the pointwise bound 
   	$$
   |\p_z[D^2f(\phi_{e^{-\f{s}{2} \Ga}}) V^2]|\leq C [|D^3f(\phi_{e^{-\f{s}{2} \Ga}})| |\phi'| |V|^2+ 
    |D^2 f(\phi_{e^{-\f{s}{2} \Ga}})|  |V| |\p_z V|]. 
   	$$
   	Due to the Taylor's  remainder formula, we can represent the error  term as follows 
   	$$
   	e^{2s} E(e^{-s} V)=\f{e^{-s}}{6} \int_0^1 D^3f(\phi_{e^{-\f{s}{2} \Ga}} +p e^{-s} V)V^3 (1-p)^3 dp, 
   	$$
   	whence by taking into account that $f\in C^4$ and $\phi, \phi', V$ are bounded functions, we have the pointwise bound 
   	\begin{equation}
   	\label{560} 
   	|\p_z e^{2 s} E(e^{-s} V)|\leq C e^{-s} [|\p_z V| |V|^2+ 
   	|V|^3 |\phi'_{e^{-\f{s}{2} \Ga}}|+|\p_z V| |V|^3 e^{-s}]. 
   	\end{equation}
   	Altogether, we get  the pointwise bounds  
   	$| H[\Ga, V]|+  	|\p_z[H[\Ga, V]| \leq  C [|V|^2 +   |V| |\p_z V|]$.  So, by \eqref{550} and \eqref{500}, we conclude 
   	$$
   	\|H(\Ga, V)(s)\|_{L^2_\eta(m) H^1_z }\leq 
   	C  \|V\|_{L^\infty_{z, \eta}} [\|V\|_{L^2(m) L^2_z}+ \|\p_z V\|_{L^2(m) L^2_z}]\leq C\eps^2 e^{-(2 n-1) s}. 
   	$$
   	Next, we deal with $N_2 (\Ga, \nabla_{\eta} \Ga, V)$. Recall 
   	 	\begin{eqnarray*}
    	 N_2 (\Ga, \nabla_{\eta}\cdot \Ga, V) &=&  K_1 (e^{- \f{s}{2}} \Ga) (\nabla_{\eta} \cdot \Ga)^2+ \f{1}{2 }  K_2(e^{- \f{s}{2}} \Ga)     D^2 f(\phi_{e^{- \f{s}{2}} \Ga}) \langle V^2, \psi\rangle  \\
   	 &+& 
   	  \f{1}{2 }  K_2(e^{- \f{s}{2}} \Ga)\left(e^{2s}  \dpr{\psi}{E(e^{-s} V)} +2 e^{s}  \langle \psi, (D f(\phi_{e^{- \f{s}{2}}\Ga})- D f(\phi) ) V\rangle \right). 
  	\end{eqnarray*}
  	Before we get on with $N_2$, recall that $|K_1(\si)|=O(1), |K_2(\si)|=O(1)$. Thus, 
  	 $  	 |K_1 (e^{- \f{s}{2}} \Ga) (\nabla_{\eta}\cdot  \Ga)^2|\leq C   |\nabla_\eta \Ga|^2$. 
We have by \eqref{510}, 
   	\begin{eqnarray*}
   	\|K_1 (e^{- \f{s}{2}} \Ga) (\nabla_{\eta} \cdot  \Ga)^2\|_{L^2(m)} \leq C   \|\nabla_\eta \cdot  \Ga\|_{L^2(m)}\|\nabla_\eta \cdot  \Ga\|_{L^\infty_\eta}  \leq C \eps^2 e^{-(n-2)s}
   	\end{eqnarray*}
   	Regarding the other terms, we estimate away the term $K_2(e^{-\f{s}{2}})$ by a constant and 
   		\begin{eqnarray*}
   		& &   	\| D^2 f(\phi_{e^{- \f{s}{2}} \Ga}) \langle V^2, \psi\rangle \|_{L^2(m)} +2 e^{s} \| \langle \psi, (D f(\phi_{e^{- \f{s}{2}}\Ga})- D f(\phi) ) V\rangle \|_{L^2(m)} + \\
   			&+& 
   			e^{2s} \| \dpr{\psi}{E(e^{-s} V)}\|_{L^2(m)}  \leq  C \|V\|_{L^2_z L^2(m)} \|V\|_{L^\infty_{\eta, z}} + C e^{\f{s}{2}} \|V\|_{L^\infty_\eta L^2_z} \|\Ga\|_{L^2(m)}  + \\
   			&+& C e^{-s} \|V\|_{L^\infty_{\eta, z}}^2 \|V\|_{L^2_z L^2(m)}\leq C\eps^2 e^{-\f{3n-4}{2}s}\leq C \eps^2 e^{-(n-2)s}.
   		\end{eqnarray*}
   	For the estimate of $\|N_2 (\Ga, \nabla_{\eta} \Ga, V)\|_{L^\infty_\eta}$, we have 
   	 	\begin{eqnarray*}
   	 		\|K_1 (e^{- \f{s}{2}} \Ga) (\nabla_{\eta} \cdot  \Ga)^2\|_{L^\infty_\eta}  \leq  C 
   	 	 	\|\nabla_\eta  \Ga\|_{L^\infty_\eta}^2\leq C \eps^2 e^{-(n-2)s}. 
   	 	\end{eqnarray*}
   	For the other terms 
   	\begin{eqnarray*}
   		& &   	\| D^2 f(\phi_{e^{- \f{s}{2}} \Ga}) \langle V^2, \psi\rangle \|_{L^\infty_\eta} +2 e^{s} \| \langle \psi, (D f(\phi_{e^{- \f{s}{2}}\Ga})- D f(\phi) ) V\rangle \|_{L^\infty_\eta} +  
   		e^{2s} \| \dpr{\psi}{E(e^{-s} V)}\|_{L^\infty_\eta}  \leq   \\
   		&\leq & C  \|V\|_{L^\infty_{\eta,z}} \|V\|_{L^\infty_{\eta} L^2_z}+C e^{\f{s}{2}} \|V\|_{L^\infty_\eta L^2_z} \|\Ga\|_{L^\infty_\eta} + C e^{-s}  \|V\|_{L^\infty_{\eta,z}} ^2 \|V\|_{L^\infty_\eta L^2_z} \leq C\eps^2 e^{-\f{3n-4}{2}s}\leq C\eps^2 e^{-(n-2)s}. 
   	\end{eqnarray*}
   	This completes the analysis of $N_2 (\Ga, \nabla_{\eta} \Ga, V)$ and \eqref{530} is established.

   	Finally, we discuss the proof of \eqref{540}, that is the control of the $N_1$ term in the relevant norms.    	Recall 
   	$$
   	 Q_0 N_1(\Ga, \nabla_{\eta} \cdot  \Ga, V)=  N_2(\Ga, \nabla_{\eta} \cdot \Ga, V) Q_0[\phi'_{e^{-\f{s}{2}} \Ga}]  + Q_0[e^{s} \big( D f(\phi_{e^{- \f{s}{2}}\Ga})- D f(\phi) \big) V+ 
   	 e^{- \f{s}{2}} (\nabla_{\eta}\cdot   \Ga)^2 \phi''_{e^{-\f{s}{2}} \Ga}]. 
   	$$
   	For the first term, 	note that since $Q_0[\phi']=0$ and \eqref{510}, 
   	\begin{eqnarray*}
   		\|Q_0[\phi'_{e^{-\f{s}{2}} \Ga}]\|_{H^1_z}=\|Q_0[\phi'_{e^{-\f{s}{2}} \Ga}-\phi']\|_{H^1_z}\leq C e^{-\f{s}{2}} \|\Ga\|_{L^\infty}\leq C \eps e^{-\f{n-1}{2}s}. 
   	\end{eqnarray*} 
   	We thus easily have by \eqref{530}, 
 \begin{eqnarray*}
   	\| N_2(\Ga, \nabla_{\eta}  \cdot \Ga, V) Q_0[\phi'_{e^{-\f{s}{2}} \Ga}]\|_{L^2(m)H^1_z\cap L^\infty_\eta H^1_z} &\leq &  C \|N_2(\Ga, \nabla_{\eta}  \cdot \Ga, V)\|_{L^2(m)\cap L^\infty_\eta} \|Q_0[\phi'_{e^{-\f{s}{2}} \Ga}]\|_{H^1_z} \\
   		&\leq& C\eps^3 e^{-\f{3n-5}{2}s}. 
\end{eqnarray*} 
   	For the next term, we use the boundedness of $Q_0$ in the function spaces that we use, to conclude  
 \begin{eqnarray*}
 \|e^{s} \big( D f(\phi_{e^{- \f{s}{2}}\Ga})- D f(\phi) \big) V\|_{L^2(m) H^1_z\cap L^\infty_\eta H^1_z} 
 &\leq &  C e^{\f{s}{2}} [\|\Ga\|_{L^2(m)}+ \|\Ga\|_{L^\infty_\eta}] (\|V\|_{L^\infty_\eta H^1_z}+\|V\|_{L^\infty_{\eta,z}})   \leq \\
 &\leq & C\eps^2 e^{-\f{3n-4}{2}s}\leq C \eps^2 e^{-(n-\f{3}{2}) s}.
 \end{eqnarray*}
   	For the last term, we have 
   	\begin{eqnarray*}
   & & 	 \|e^{- \f{s}{2}} (\nabla_{\eta}  \cdot \Ga)^2 \phi''_{e^{-\f{s}{2}} \Ga }\|_{L^2(m) H^1_z} \leq
   	  C e^{- \f{s}{2}} \|\nabla_\eta \cdot \Ga\|_{L^\infty_\eta} \|\nabla_\eta \cdot \Ga\|_{L^2(m)}\leq C\eps^2 e^{-(n-\f{3}{2})s}, \\
   	   & & 	 \|e^{- \f{s}{2}} (\nabla_{\eta}  \cdot \Ga)^2 \phi''_{e^{-\f{s}{2}} \Ga }\|_{L^\infty_\eta H^1_z} \leq
   	   C e^{- \f{s}{2}} \|\nabla_\eta \Ga\|_{L^\infty_\eta}^2 \leq C\eps^2 e^{-(n-\f{3}{2})s}. 
   	\end{eqnarray*}
   	Putting everything together, we arrive at \eqref{540}. Note that for $n\geq 3$, the dominant decay 
   	term for $e^{-(n-\f{3}{2})s}$ came only from the contribution of the term 
   	$Q_0[e^{- \f{s}{2}} (\nabla_{\eta} \cdot  \Ga)^2 \phi''_{e^{-\f{s}{2}} \Ga }]=
   	e^{- \f{s}{2}} (\nabla_{\eta}  \Ga)^2 Q_0[\phi''_{e^{-\f{s}{2}} \Ga }]$, since\footnote{Since $\phi'$ is the eigenvector for the single eigenvalue at zero for $L_1$, we have that $Q_0[g]\neq 0$ for all $g\neq \phi'$}   $Q_0[\phi'']\neq 0$. 
   	For $n=2$, the decay terms $e^{-\f{3n-5}{2}}=e^{-(n-\f{3}{2})s}=e^{-\f{s}{2}}$, so two terms contribute at the same rate. Even in this case though, the contribution of $N_2(\Ga, \nabla_{\eta} \cdot \Ga, V) Q_0[\phi'_{e^{-\f{s}{2}} \Ga}] $ is of order $\eps^3 e^{-s/2}$ versus $\eps^2 e^{-s/2}$  of 
   	$Q_0[e^{- \f{s}{2}} (\nabla_{\eta}\cdot   \Ga)^2 \phi''_{e^{-\f{s}{2}} \Ga }]$. 
   \end{proof}
   
   \subsection{Estimates on the Duhamel's terms}
 The following elementary lemmas will be useful as well. 
 \begin{lemma}
 	\label{ex0}
 		If $c,d>0: c\neq d$, then 
 		\begin{equation}
 		\label{714}
 		\int_0^{\tau} e^{-d(\tau-s)}\left( \f{1}{\sqrt{\tau-s}} +1\right)e^{-c s} ds\leq C_{c,d} e^{-\min(c,d)\tau}. 
 		\end{equation}
 	Let $b \in \R$, $\de> 0$ and $c \geq 0$ then
 	\begin{equation}
 	\label{711} 
 	\int_0^{\tau} e^{b (\tau- s)} e^{- \de (e^{\tau} - e^s)} e^{- c s} ds \leq C_{b, \de} e^{- (c+ 1) \tau}.
 	\end{equation}
 
 \end{lemma}
The proof of Lemma \ref{ex0} is postponed for the Appendix. 
    We are now ready to deal with the Duhamel's term contributions, that is estimates \eqref{310}. 
    \subsubsection{The Duhamel's portion of $\al(z,\tau)$ in \eqref{200}}
We have by \eqref{20_5} 
  \begin{eqnarray*}
 & &  \|\int_0^{\tau} e^{- \f{n- 3}{2} (\tau- s)} e^{(e^{\tau}- e^{s})  L_1} Q_0 \Bigg[ \langle  H(\Ga, V), 1 \rangle_{\eta}(s) + \langle   N_1 (\Ga, \nabla_{\eta} \cdot \Ga, V), 1 \rangle_{\eta}(s) \Bigg] ds\|_{H^1_z}\leq \\
  & \leq & C \int_0^{\tau} e^{- \f{n- 3}{2} (\tau- s)} e^{-\f{\de}{2}(e^\tau- e^s)}  
   [\| \langle  H(\Ga, V), 1 \rangle_{\eta}(s)\|_{H^1_z} + \|\langle   N_1 (\Ga, \nabla_{\eta}\cdot  \Ga, V), 1 \rangle_{\eta}(s)\|_{H^1_z} ]ds \leq \\
   &\leq & C \int_0^{\tau} e^{- \f{n- 3}{2} (\tau- s)} e^{-\f{\de}{2}(e^\tau- e^s)}  [\|H(\Ga, V)(s)\|_{H^1_zL^2_\eta(m) } +   \|N_1 (\Ga, \nabla_{\eta} \cdot  \Ga, V)(s)\|_{H^1_zL^2_\eta(m) }] ds 
  \end{eqnarray*}
  According to \eqref{520} and \eqref{540}, the last expression is controlled by 
  $$
  C\eps^2  \int_0^{\tau} e^{- \f{n- 3}{2} (\tau- s)} e^{-\f{\de}{2}(e^\tau- e^s)}e^{-(n-\f{3}{2})s} ds\leq C \eps^2 e^{-(n-\f{1}{2})\tau},
  $$
  where in the last step, we have used \eqref{711}. 
   
    \subsubsection{The Duhamel's portion of $\ga(\tau)$ in \eqref{205}}
   $$
   	\int_0^{\tau} e^{- \f{n- 2}{2} (\tau- s)} e^{-\f{s}{2}} | \langle  
   		N_2 (\Ga, \nabla_{\eta} \cdot  \Ga, V), 1 \rangle_{\eta}(s) | ds \leq 	
   	C 	\int_0^{\tau} e^{- \f{n- 2}{2} (\tau- s)} e^{-\f{s}{2}} 
   		\|N_2 (\Ga, \nabla_{\eta} \cdot  \Ga, V), 1 \rangle_{\eta}(s)\|_{L^2(m)}  ds 
   $$ 
  The last expression is controlled, in view of \eqref{530}, by 
  $$
  C \eps^2 	\int_0^{\tau} e^{- \f{n- 2}{2} (\tau- s)} e^{-\f{s}{2}} 
    e^{-(n-2)s}  ds \leq C \eps^2 e^{- \f{n- 2}{2} \tau}. 
  $$
   \subsubsection{The Duhamel's portion of $\tilde{V}$ in \eqref{210}}
   We first take the norm $\|\cdot\|_{L^2(m)H^1_z}$. Let $l\in \{0,1\}$. We obtain from \eqref{303}, \eqref{1_9} and \eqref{20_5} and Fubini's 
   \begin{eqnarray*}
  & &  \| \int_0^{\tau} e^{(\tau- s)(\cl_{\eta}+ \f{1}{2}) } \mathcal{Q}_0 e^{(e^{\tau}- e^s) L_1} Q_0 
    \Big[  H(\Ga, V)(s)+   N_1 (\Ga, \nabla_{\eta} \cdot \Ga, V)(s)\Big] ds\|_{L^2(m) H^1_z} =\\
    &= &  \| \int_0^{\tau}      e^{(\tau- s)(\cl_{\eta}+ \f{1}{2}) }  \cq_0 \nabla_z^le^{(e^{\tau}- e^s) L_1} Q_0 
    \Big[  H(\Ga, V)(s)+   N_1 (\Ga, \nabla_{\eta} \cdot \Ga, V)(s)\Big]  \|_{L^2_z L^2_\eta(m)} ds\\
    &\leq &  \int_0^{\tau}  e^{-\f{n-2}{2}(\tau-s)} \|\nabla_z^le^{(e^{\tau}- e^s) L_1} Q_0 
    \Big[  H(\Ga, V)(s)+   N_1 (\Ga, \nabla_{\eta}\cdot  \Ga, V)(s)\Big] \|_{L^2_z L^2_\eta(m)} ds \leq \\
    &\leq &  C \int_0^{\tau}  e^{-\f{n-2}{2}(\tau-s)} e^{-\f{\de}{2}(e^\tau-e^s)} 
    [\|H(\Ga, V)(s)\|_{H^1_zL^2_\eta(m) } +   \|N_1 (\Ga, \nabla_{\eta} \cdot \Ga, V)(s)\|_{H^1_zL^2_\eta(m) }] ds 
   \end{eqnarray*}
  
  Next, we deal with $\|\cdot\|_{L^\infty_\eta H^1_z}$. We have from \eqref{20_55}   
   \begin{eqnarray*}
   	& &  \| \int_0^{\tau} e^{(\tau- s)(\cl_{\eta}+ \f{1}{2}) } \mathcal{Q}_0 e^{(e^{\tau}- e^s) L_1} Q_0 
   	\Big[  H(\Ga, V)(s)+   N_1 (\Ga, \nabla_{\eta} \cdot \Ga, V)(s)\Big] ds\|_{L^\infty_\eta H^1_z} =\\
   	&= &  \| \int_0^{\tau}      e^{(\tau- s)(\cl_{\eta}+ \f{1}{2}) } \cq_0  \nabla_z^le^{(e^{\tau}- e^s) L_1} Q_0 
   	\Big[  H(\Ga, V)(s)+   N_1 (\Ga, \nabla_{\eta}  \cdot \Ga, V)(s)\Big]  \|_{L^\infty_\eta L^2_z  } ds\\
   	&\leq & \int_0^\tau e^{-\f{n-2}{2}(\tau-s)} \|\nabla_z^le^{(e^{\tau}- e^s) L_1} Q_0 
   	\Big[  H(\Ga, V)(s)+   N_1 (\Ga, \nabla_{\eta}\cdot  \Ga, V)(s)\Big]  \|_{L^\infty_\eta L^2_z \cap L^2(m) L^2_z }ds \leq \\
   		&\leq & \int_0^\tau e^{-\f{n-2}{2}(\tau-s)} e^{-\f{\de}{2}(e^\tau-e^s)}  [\|H(\Ga, V)(s)\|_{H^1_zL^2_\eta(m)\cap L^\infty_\eta L^2_z } +   \|N_1 (\Ga, \nabla_{\eta} \cdot \Ga, V)(s)\|_{H^1_zL^2_\eta(m) \cap L^\infty_\eta L^2_z }] ds 
   \end{eqnarray*}
  In view of \eqref{520} and \eqref{540}, we control both contributions by 
  $$
  C\eps^2 \int_0^\tau e^{-\f{n-2}{2}(\tau-s)} e^{-\f{\de}{2}(e^\tau-e^s)} e^{-(n-\f{3}{2})s} ds\leq C \eps^2 e^{-(n-\f{1}{2})\tau}, 
  $$
   where again  in the last step, we have used \eqref{711}. 
  \subsubsection{The Duhamel's portion of  \ $\tilde{\Ga}$ in \eqref{215}}
  For $l\in \{0,1\}$, we obtain from \eqref{1_9} 
  $$
   \|\int_0^{\tau} e^{(\tau- s) \cl_{\eta} }  \mathcal{Q}_0  e^{-\f{s}{2}}  N_2 (\Ga, \nabla_{\eta} \cdot \Ga, V) (s) ds  \|_{H^1(m)} \leq  C \int_0^\tau e^{-\f{n-1}{2}(\tau-s)} e^{-\f{s}{2}}  
\|N_2 (\Ga, \nabla_{\eta} \cdot \Ga, V) (s)\|_{L^2(m)} ds 
 $$ 
  Next, for the norm $\|\cdot\|_{L^\infty_\eta}$, we obtain from \eqref{20_55} 
  $$
    \|\int_0^{\tau} e^{(\tau- s) \cl_{\eta} }  \mathcal{Q}_0  e^{-\f{s}{2}}  N_2 (\Ga, \nabla_{\eta}\cdot  \Ga, V) (s) ds  \|_{L^\infty_\eta} \leq  C \int_0^\tau  e^{-\f{n-1}{2}(\tau-s)}  e^{-\f{s}{2}}  
   	\|N_2 (\Ga, \nabla_{\eta} \cdot \Ga, V) (s)\|_{L^2(m)\cap L^\infty} ds 
  $$
  Finally, for $\|\nabla[\cdot]\|_{L^\infty_\eta}$, we obtain from \eqref{20_55}  
  $$
  \|\int_0^{\tau} \nabla_\eta e^{(\tau- s) \cl_{\eta} }  \mathcal{Q}_0  e^{-\f{s}{2}}  N_2 (\Ga, \nabla_{\eta}  \cdot \Ga, V) (s) ds  \|_{L^\infty_\eta} \leq  C \int_0^\tau \f{e^{-\f{n-1}{2}(\tau-s)}}{\sqrt{a(\tau-s)}}  e^{-\f{s}{2}}  
  \|N_2 (\Ga, \nabla_{\eta} \cdot \Ga, V) (s)\|_{L^2(m)\cap L^\infty} 
  $$
  By \eqref{530}, we control the last three integrals by 
  $$
  C\eps^2 [\int_0^\tau \f{e^{-\f{n-1}{2}(\tau-s)}}{\sqrt{\tau-s}}  e^{-\f{s}{2}}  
   e^{-(n-2)s} ds +  \int_0^\tau  e^{-\f{n-1}{2}(\tau-s)}   e^{-\f{s}{2}}  
   e^{-(n-2)s} ds] \leq C\eps^2 e^{-\f{n-1}{2}\tau},
  $$
   where in the last stage, we have used \eqref{714}.

  \section{Sharpness of the decay rates and asymptotic profiles} 
  \label{sec:5} 
In this section,  we discuss the sharpness of these rates as well as the asymptotic profiles. 

\subsection{The asymptotic profiles for $\si$} The statements for $\Ga$ are straightforward as the decay rate for $\ga(\tau)$ (see \eqref{700}), $e^{-\f{n-2}{2}\tau}$ is strictly slower   than the decay rate for $\tilde{\Ga}$, which is $e^{-\f{n-1}{2}\tau}$. In addition, by examining the evolution equation for $\ga(\tau)$, 
\eqref{205} and the subsequent estimates in Section \ref{sec:4}, we see that 
$$
\ga(\tau)=\ga(0) e^{-\f{n-2}{2}\tau}+O(e^{-\f{n-1}{2}\tau})= \dpr{\Ga(0, \cdot)}{1}_\eta e^{-\f{n-2}{2}\tau}+O(e^{-\f{n-1}{2}\tau})= (\int_{\rne} \si_0(y) dy) e^{-\f{n-2}{2}\tau}+O(e^{-\f{n-1}{2}\tau}).
$$ 
   It follows that 
   $$
   \|\Ga(\tau, \cdot) - (\int_{\rne} \si_0(y) dy) e^{-\f{n-2}{2}\tau}G(\cdot)\|_{L^\infty_\eta}\leq C\eps^2 e^{-\f{n-1}{2}\tau}.
   $$
  By the estimates for $\nabla_\eta \tilde{\Ga}$ in $L^\infty_\eta$, it follows that 
  $$
  \|\nabla[\Ga(\tau, \cdot) - (\int_{\rne} \si_0(y) dy) e^{-\f{n-2}{2}\tau} G(\cdot)]\|_{L^\infty_\eta}\leq C\eps^2 e^{-\f{n-1}{2}\tau}.
  $$
  Translating back to the original variables, 
 \begin{eqnarray*}
& &   \left\|\si(t, \cdot)-  \f{(\int_{\rne} \si_0(y) dy)}{(1+t)^{\f{n-1}{2}}}  G\left(\f{\cdot}{\sqrt{1+t}}\right) \right\|_{L^\infty_y}\leq \f{C\eps^2}{(1+t)^{\f{n}{2}}}, \\
& &   \left\|\nabla_y\si(t, \cdot)-  \f{(\int_{\rne} \si_0(y) dy)}{(1+t)^{\f{n}{2}}}  
(\nabla_y G)\left(\f{\cdot}{\sqrt{1+t}}\right) \right\|_{L^\infty_y}\leq \f{C\eps^2}{(1+t)^{\f{n+1}{2}}}, 
 \end{eqnarray*}
 These are precisely the estimates \eqref{71}, \eqref{72}. 
 \subsection{Asymptotic profiles for the radiation term $v$} 
  The goal in this section is to isolate a leading order term, $\bar{V}$ for $V$, which decays at the leading order rate $e^{-(n-\f{1}{2})\tau}$. A quick look at the estimates for the free solutions in Section \ref{sec:4.1} confirms that they decay exponentially in $e^\tau$. 
  
  Next, going to the Duhamel terms, assume for the moment $n\geq 3$.  We have seen that the leading order nonlinearity is exactly $Q_0 [e^{-\f{s}{2}} (\nabla_\eta\cdot \Ga)^2 \phi''_{e^{-\f{s}{2}} \Ga}]$, which decays of the order $e^{-(n-\f{3}{2})s}$ (and thus produces through the Duhamel's operator an object with a decay of about $e^{-(n-\f{1}{2})\tau}$), while all the others are of rates of at least $e^{-\f{3n-5}{2}s}$ (and thus produce, through the Duhamels operator terms of decay of at least $e^{-\f{3 n-3}{2}\tau}$). {\it Note that in this argument, we certainly need to establish lower bound for the Duhamel's operator, which is  acting on what we believe is the main term, $Q_0 [e^{-\f{s}{2}} (\nabla_\eta\cdot \Ga)^2 \phi''_{e^{-\f{s}{2}} \Ga}]$. So far, we have only established upper bounds and it is not clear {\it a priori} whether some hidden cancellation does not occur within the Duhamel's operator formalism.} 
  
  In order to establish the said lower bounds, we start by further reducing the leading order terms, by peeling off lower order (i.e. faster decaying) terms.  Taking into account $\tilde{\Ga}=O(e^{-\f{n-1}{2}s})$ and $e^{-\f{s}{2}} \Ga=O(e^{-\f{n-1}{2}s})$, 
   \begin{eqnarray*}
    & & Q_0 [e^{-\f{s}{2}} (\nabla_\eta\cdot \Ga)^2 \phi''_{e^{-\f{s}{2}} \Ga}] =  e^{-\f{s}{2}} (\nabla_\eta\cdot \Ga)^2 Q_0[\phi''_{e^{-\f{s}{2}} \Ga}]= e^{-\f{s}{2}} (\nabla_\eta\cdot (\ga(\tau) G+\tilde{\Ga}))^2
    Q_0[\phi''+(\phi''_{e^{-\f{s}{2}} \Ga}-\phi'')]\\
    &=& e^{-\f{s}{2}} (\nabla_\eta\cdot (\ga(s) G))^2
    Q_0[\phi'']+O(e^{-(n-1)s})=\ga_0^2 e^{-(n-\f{3}{2})s}  (\nabla_\eta\cdot  G)^2 Q_0[\phi'']+O(e^{-(n-1)s})
   \end{eqnarray*}
  where in the last equality, we  used  $\ga(s)=\ga_0 e^{-\f{n-2}{2}s}+O(e^{-\f{n-1}{2}s})$. 
  In view of the equations \eqref{VV}, we see that if the term $\bar{V}$ satisfies {\it the linear inhomogeneous equation} 
  \begin{equation}
  \label{8001} 
  \bar{V}_\tau=(\cl_\eta+\f{1}{2}) \bar{V} + e^\tau L_1 \bar{V} + \ga_0^2 e^{-(n-\f{3}{2})\tau}  (\nabla_\eta\cdot  G)^2 Q_0[\phi''],  \bar{V}(0)=0.
  \end{equation}
  where we recall that $\ga_0=\dpr{\Ga}{1}_\eta=\int_{\rne} \si_0(y) dy$.  Denote $H:= (\nabla_y \cdot  e^{-\f{|y|^2}{4}})^2=\f{|y|^2}{4} e^{-\f{|y|^2}{2}}$. Then, \eqref{8001} reads 
   \begin{equation}
   \label{800} 
   \bar{V}_\tau=(\cl_\eta+\f{1}{2}) \bar{V} + e^\tau L_1 \bar{V} + \ga_0^2 e^{-(n-\f{3}{2})\tau}   Q_0[\phi''](z) H(\eta),  \bar{V}(0, z, \eta)=0.
   \end{equation}
  Due to the estimates that we had for the remaining nonlinearities (and more precisely \eqref{711}, which upgrades the Duhamel's term by $e^{-\tau}$ over the non-linearity) , 
  we will have the asymptotic estimate 
   \begin{equation}
   \label{810} 
  \|V(\tau, \cdot)-\bar{V}(\tau, \cdot)\|_{(H^1(m)\cap W^{1, \infty})_\eta H^1_z}\leq C \eps^2 e^{-n\tau}. 
 \end{equation}
 At this point, it is more advantageous to translating back to the original variables. In doing so,  via the assignment $\bar{v}(z,y,t)=\f{1}{1+t} 
  \bar{V}(z, \f{y}{\sqrt{1+t}}, \ln(1+t))$, we obtain the following equation for $\bar{v}$ 
\begin{equation}
\label{820} 
   \bar{v}_t=L \bar{v} + \f{(\int_{\rne} \si_0(y) dy)^2}{(1+t)^{n+\f{1}{2}}}     H
   	\left(\f{y}{\sqrt{1+t}}\right) Q_0[\phi''],  \bar{v}(0)=0, 
\end{equation}
where recall $L=L_1+\De_y$. 
Similarly, \eqref{810} translates into the following  estimate for $v-\bar{v}$, 
\begin{equation}
\label{830} 
\|v(t, \cdot)-\bar{v}(t, \cdot)\|_{L^\infty_{y z}}\leq C \eps^2 (1+t)^{-(n+1)}.
\end{equation}
  We will now compute $\bar{v}$ to a leading order. As a solution to \eqref{820}, we have the formula 
  \begin{eqnarray*}
  	\bar{v}(t)=c_0 \int_0^t e^{(t-s) L_1}[Q_0 \phi''] 
  	\f{e^{(t-s) \De_y}[H \left(\f{\cdot }{\sqrt{1+s}}\right)]}{(1+s)^{n+\f{1}{2}}} ds, c_0:=\f{(\int_{\rne} \si_0(y) dy)^2}{(4\pi)^{n-1}}. 
  \end{eqnarray*}
  Next, we need to compute  $e^{(t-s) \De_y}[H \left(\f{\cdot}{\sqrt{1+s}}\right)]$. 
   Before we go any further, we take a moment to introduce the Fourier transform, its inverse and some explicit formulas that will be useful. 
   $$
   \hat{f}(\xi)=\int_{\rne} f(x) e^{-2\pi i x \cdot  \xi} dx, \ \  f(x) =\int_{\rne}  \hat{f}(\xi) e^{2\pi i x\cdot \xi} d\xi
   $$
  With this definition, $\widehat{e^{-a |x|^2}}(\eta)= \left(\f{\pi}{a} \right)^{\f{n-1}{2}} 
  e^{-  \f{\pi^2 |\eta|^2}{a}}$, so 
   $$
   \hat{H}(\eta)=-\f{1}{16\pi^2} \De_\eta[\widehat{e^{-\f{|\cdot|^2}{2}}}]= \f{(2\pi)^{\f{n-1}{2}}}{4} e^{-2\pi^2 |\eta|^2}(1+c_1 |\eta|^2). 
   $$
   for some constant $c_1$.   Furthermore,  
  \begin{eqnarray*}
   \widehat{e^{(t-s) \De_y}[H\left(\f{y}{\sqrt{1+s}}\right)]}(\eta) &=& e^{-4\pi^2(t-s) |\eta|^2} (1+s)^{\f{n-1}{2}} \hat{H}( \eta \sqrt{1+s})= \\
  	&=& \f{(2\pi)^{\f{n-1}{2}}}{4}(1+s)^{\f{n-1}{2}} e^{-2\pi^2 (2t+1-s) |\eta|^2}  (1+c_1 (1+s) |\eta|^2).
  \end{eqnarray*}
    Eventually, in  the term $(1+s)^{\f{n+1}{2}} |\eta|^2e^{-2\pi^2 (2t+1 -s) |\eta|^2} $ produces lower order terms, so it can be dropped. Note that $2t+1-s>0$, when $s\in (0,t)$. Inverting the Fourier transform above yields 
    $$
    e^{(t-s) \De_y}[H\left(\f{\cdot}{\sqrt{1+s}}\right)] (y) =  \left(\f{1+s}{2t+1-s}\right)^{\f{n-1}{2}} e^{-\f{|y|^2}{2(2t+1-s)} }+ l.o.t.
    $$
    This allows us to  write 
    $$
    \bar{v}(t)=c_0  \int_0^t e^{(t-s) L_1}[Q_0 \phi'']  
    \f{e^{-\f{|y|^2}{2(2t+1-s)}} }{(2t+1-s)^{\f{n-1}{2}}(1+s)^{\f{n}{2}+1}}  ds+l.o.t. 
    $$
Introduce  $M(t,s,y):=\f{e^{-\f{|y|^2}{2(2t+1-s)}} }{(2t+1-s)^{\f{n-1}{2}}(1+s)^{\f{n}{2}+1}}$ and note that the operator $L_1$ is invertible on $Q_0[L^2_z]$. Thus, performing an integration by parts, 
  \begin{eqnarray*}
  	I(t,y,z) &=& \int_0^t M(t,s,y) e^{(t-s) L_1}[Q_0 \phi'']  ds=-
  	M(t,s,y) e^{(t-s)L_1}L_1^{-1} Q_0[\phi''] |_{0}^t + \\
  	&+&  \int_0^t e^{(t-s) L_1} [L_1^{-1} Q_0 \phi''] \f{\p M}{\p s} (t,s,y) ds=-L_1^{-1} Q_0[\phi'']  M(t,t,y) + \\
  	&+& 
  	M(t,0,y) e^{t L_1} [L_1^{-1} Q_0[\phi'']]+\int_0^t e^{(t-s) L_1} [L_1^{-1} Q_0 \phi''] \f{\p M}{\p s} (t,s,y) ds. 
  \end{eqnarray*}
  We argue that the leading order  term is 
  \begin{equation}
  \label{850} 
  -c_0 L_1^{-1} Q_0[\phi'']  M(t,t,y)= - c_0 \f{e^{-\f{|y|^2}{2(t+1)}}}{(t+1)^{n+\f{1}{2}}} L_1^{-1} Q_0[\phi''], 
  \end{equation}
  which clearly has a decay rate in $L^\infty_{y,z}$ of order $(1+t)^{-(n+\f{1}{2})}$ as stated.  We now need to show that the remaining two terms have faster decay rates.  For the term $e^{t L_1} [L_1^{-1} Q_0 \phi''] $, we have by Sobolev embedding and \eqref{20_5}
 \begin{equation}
 \label{814} \|e^{t L_1} [L_1^{-1} Q_0 \phi''] \|_{L^\infty_z}\leq C   \|e^{t L_1} [L_1^{-1} Q_0 \phi''] \|_{H^1_z}\leq C_\de e^{-\f{\de}{2} t} \|L_1^{-1} Q_0 \phi'']\|_{H^1_z}, 
 \end{equation}
  so it has an exponential decay in time.  Similarly, splitting  the integral 
  $$
    \int_0^t e^{(t-s) L_1} [L_1^{-1} Q_0 \phi''] \f{\p M}{\p s} (t,s,y) ds=\int_0^{t-\sqrt{t}}   \ldots ds + \int_{t-\sqrt{t}}^t \ldots ds
    $$
     allows us to estimate the former integral as follows, 
\begin{eqnarray*}
   \| \int_0^{t-\sqrt{t}}  e^{(t-s) L_1} [L_1^{-1} Q_0 \phi''] \f{\p M}{\p s} (t,s,y) ds\|_{L^\infty_z} &\leq &   \int_0^{t-\sqrt{t}}  \|e^{(t-s) L_1} [L_1^{-1} Q_0 \phi'']\|_{L^\infty_z}  |\f{\p M}{\p s} (t,s,y)| ds   \\
 &\leq & C_\de e^{-\f{\de}{2} \sqrt{t}}  \|L_1^{-1} Q_0 \phi'']\|_{H^1_z}\leq C (1+t)^{-(n+1)}. 
 \end{eqnarray*}
 since on the region of integration $t-s\geq \sqrt{t}$, and we can apply \eqref{814}.  
 For the latter integral, one can see that for $s\in (t-\sqrt{t},t)$, we have by \eqref{814},  $\|e^{(t-s) L_1} [L_1^{-1} Q_0 \phi'']\|_{L^\infty_z}\leq C_\de$, so that 
  \begin{eqnarray*}
     \| \int_{t-\sqrt{t}}^t  e^{(t-s) L_1} [L_1^{-1} Q_0 \phi''] \f{\p M}{\p s} (t,s,y) ds\|_{L^\infty_{z,y}} & \leq &   \int_{t-\sqrt{t}}^t  \|e^{(t-s) L_1} [L_1^{-1} Q_0 \phi'']\|_{L^\infty_z}  \|\f{\p M}{\p s} (t,s,y)\|_{L^\infty_y} ds \\
  	&\leq & C_\de  \int_{t-\sqrt{t}}^t \|\f{\p M}{\p s} (t,s,y)\|_{L^\infty_y} ds\leq   \f{C}{(1+t)^{n+1}}, 
  \end{eqnarray*}
  where in the last step, we have used  that if $s\sim t$, then 
  $\|\f{\p M}{\p s} (t,s,y)\|_{L^\infty_y}\leq   C(1+t)^{-n-\f{3}{2}}$. 
  All in all, summarizing the results from this section, we have established that 
  $$
  \|\bar{v}+c_0 \f{e^{-\f{|y|^2}{2(t+1)}}}{(t+1)^{n+\f{1}{2}}} L_1^{-1} Q_0[\phi''] \|_{L^\infty_{z,y}}\leq C(1+t)^{-n-1},
  $$ 
  which combined with \eqref{830} leads us to \eqref{831}. 
  
  For the case of $n=2$, we saw that there are two terms in the nonlinearity (for the equation in the scaled variables)  with dominant decay rate, namely $N_2(\Ga, \nabla_{\eta} \cdot \Ga, V) Q_0[\phi'_{e^{-\f{s}{2}} \Ga}] $ and \\ $Q_0 [e^{-\f{s}{2}} (\nabla_\eta\cdot \Ga)^2 \phi''_{e^{-\f{s}{2}} \Ga}]$. We have just analyzed the second one, which produces (on a solution level and in the standard variables) the term found in \eqref{850}, which is of order $\eps^2 (1+t)^{-\f{5}{2}}$, for $n=2$.  On the other hand, the term $N_2(\Ga, \nabla_{\eta} \cdot \Ga, V) Q_0[\phi'_{e^{-\f{s}{2}} \Ga}] $  produces a solution less than $C \eps^3 (1+t)^{-\f{5}{2}}$, and as such is lower order  in $\eps$, but of the same order in terms of power decay in $t$.  These exact results are summarized in \eqref{831} and  \eqref{832}. 
  \appendix
  
  \section{Proof of Lemma \ref{le:10}} 
Set up a mapping
$$
{\mathbb G}(w; v, \si)(z,y)=\phi(z-\si(y))+v(z,y) -\phi(z)-w(z,y)
$$
We will show first that ${\mathbb G}: (H^1(m)\cap W^{1, \infty})_y H^1_{z} \times {\mathcal R}\times  (H^1(m) \cap W^{1, \infty})   \to (H^1(m)\cap W^{1, \infty})_y H^1_{z}$. This follows easily from the mean value theorem, since 
$$
{\mathbb G}(w; v, \si)(z,y)=-\si(y) \int_0^1 \phi'(z-\tau \si(y)) d\tau + v(z,y) - w(z,y), 
$$
and $\phi'\in H^1(\rone)$. 
Clearly ${\mathbb G}(0,0,0)=0$, so by the implicit function theorem,  it remains to check that 
$$
d {\mathbb G}(0,0,0) (\tilde{\si}, \tilde{v})=-\phi'(z) \tilde{\si}+\tilde{v}
$$
 is an isomorphism on  $(H^1(m)\cap W^{1, \infty})_y H^1_{z}$. To this end, let $h\in (H^1(m)\cap W^{1, \infty})_y H^1_{z}$ be an arbitrary element   and we have to resolve the equation 
 \begin{equation}
 \label{60}
 -\phi'(z) \tilde{\si}+\tilde{v}=h.
 \end{equation}
 Clearly, by the properties of ${\mathcal R}$ and ${\mathcal N}$, \eqref{60} has an unique solution, namely $\tilde{\si}(y) =-\dpr{h(\cdot, y)}{\psi(\cdot)}$, while $\tilde{v}=Q_0 h\in {\mathcal R}$. Moreover, these mappings are linear and 
 \begin{eqnarray*}
 & &  \|\tilde{\si}\|_{H^1(m) \cap W^{1, \infty} }\leq \|\psi\|_{L^2_z} \|h\|_{(H^1(m) \cap W^{1, \infty})_y L^2_{z}},\\
 & &  \|\tilde{v}\|_{H^1(m) \cap W^{1, \infty} H^1_{z}}\leq C \|h\|_{(H^1(m) \cap W^{1, \infty})_y H^1_{z}}. 
 \end{eqnarray*}
 Thus, the implicit function theorem applies and in a neighborhood of zero, there are unique and small $\si(w)\in  H^1(m) \cap W^{1, \infty} ,  v(w)\in {\mathcal R}$, so that 
 ${\mathbb G}(w; v(w), \si(w))=0$. Equivalently, \eqref{70} holds.

  \section{Proof of Lemma \ref{le:l1}} 
  The proof of the bound \eqref{20_5} follows from the Gearheart-Pr\"uss theorem in the following way. Since, by our assumption \eqref{20} the spectrum is to the left of any vertical line in the complex plane $\{z: \Re z=-\de_1\}$, $0<\de_1<\de$, it will suffice to show that for a fixed such $\de_1$, 
  \begin{equation}
  \label{65} 
 \sup_{\mu\in \rone}  \|(L_1+\de_1+ i \mu)^{-1}\|_{H^1(\rone)\to H^1(\rone)}=C_{\de_1}<\infty. 
  \end{equation}
  Indeed, the Geraheart Pr\"uss theorem guarantees that if $\si(L_1)\subset \{z: \Re z<-\de_1\}$ and \eqref{65} 	 holds,  then the operator $L_1+\de_1$ generates a semigroup with strictly negative growth bound, that is - there exists $\eps>0$, so that $\|e^{s(L_1+\de_1)}\|_{H^1(\rone)\to H^1(\rone)}\leq C_{\de_1}  e^{-\eps s}$ or, equivalently 
  $$
  \|e^{sL_1}\|_{H^1(\rone)\to H^1(\rone)}\leq C_{\de_1}  e^{- s(\eps+\de_1)}\leq C_{\de_1}  
  e^{- s \de_1},
  $$  
  which is \eqref{20_5}. 
  
  Thus, it suffices to establish \eqref{65}. To this end, fix $\de_1$ and observe that since 
  the resolvent $(L_1+z)^{-1}$ is analytic $B(H^1(\rone))$ valued function on $\{z: \Re z>-\de\}$, it is continuous in the same region and in particular, for each $N$, there is $C_N$, 
  $$
  \sup_{\mu\in \rone: |\mu|<N}   \|(L_1+\de_1+ i \mu)^{-1}\|_{H^1(\rone)\to 
  H^1(\rone)}=C_{\de_1, N}<\infty 
  $$
  Thus, the real issue is to establish the bounds in \eqref{65} for all large enough $\mu$. So, we setup $g\in H^1(\rone)$ and $f=(L_1+\de_1+ i \mu)^{-1} g$ or equivalently 
\begin{equation}
\label{67} 
 f''+c f'+ W f + \de_1 f+i \mu f=g,
\end{equation}
where $W=Df(\phi)$ is a bounded, real-valued  potential. 

 The existence of such an $f\in H^1(\rone)$ is not in any doubt, by the spectral assumptions, we just need {\it a posteriori}  uniform in $\mu$ estimates for it, for all large enough $\mu$. We take a dot product of \eqref{67} with $f$. Taking imaginary parts of the said dot product leads to the identity 
 $$
 \mu \|f\|^2+ c\Im \dpr{f'}{f}=\Im \dpr{g}{f}.
 $$
  Applying the Cauchy-Schwartz inequality and after some algebraic manipulations, we obtain that for every $\eps>0$, there is $C_\eps$, so that 
  $$
  \|f\|^2\leq \eps \|f\|^2 + \f{C_\eps}{\mu^2} (\|f'\|^2+\|g\|^2). 
  $$
  Selecting $\eps=\f{1}{2}$, we get the {\it a posteriori} estimate 
\begin{equation}
\label{69} 
 \|f\|^2\leq  \f{C}{\mu^2} (\|f'\|^2+\|g\|^2). 
\end{equation}
  We now take the real-part of the dot produc of \eqref{67} with $f$. We similarly obtain for every $\eps>0$,
  $$
  \|f'\|^2\leq \eps  \|f'\|^2 + D_\eps [\|f\|^2+\|g\|^2].
  $$ 
  Plugging in \eqref{69} into this last inequality yields 
  $$
  \|f'\|^2\leq  \eps  \|f'\|^2 +  \f{M_\eps}{\mu^2} (\|f'\|^2+\|g\|^2)+D_\eps \|g\|^2.
  $$
 Selecting $\eps=\f{1}{4}$ and then $\mu$ so large so that $\f{M_\eps}{\mu^2}<\f{1}{4}$, we arrive at
 $$
 \|f'\|^2\leq D \|g\|^2.
 $$
 Combining the last estimate with \eqref{69} yields   the desired, uniform in $\mu$ estimate \eqref{65}.

  \section{Proof of Lemma \ref{ex0}} 
  \subsection{Proof of \eqref{714}} The estimate \eqref{714} is standard. We estimate the integrals $\int_0^{\tau-1}\tau... ds$ and $\int_{\tau-1}^\tau..  ds$ separately. We have that 
  $$
  	\int_0^{\tau-1}  e^{-d(\tau-s)}\left(\f{1}{\sqrt{\tau-s}} +1 \right)e^{-c s} ds\leq e^{-d\tau} \left(\f{e^{(d-c)(\tau-1)}-1}{d-c} \right)\leq \f{e^{-\min(d,c)\tau}}{|d-c|}. 
  $$
  For the other term, 
  $$
  \int_{\tau-1}^\tau \f{e^{-d(\tau-s)}}{\sqrt{\tau-s}} e^{-c s} ds \leq e^c  e^{-c\tau}\int_{\tau-1}^\tau  \f{1}{\sqrt{\tau-s}} ds\leq e^c e^{-c\tau} \leq e^c  e^{-\min(d,c)\tau}. 
  $$
  \subsection{Proof of \eqref{711}}  Since $\lim_{h\to 0+} \f{e^h-1}{h}=1$, fix  $h_0>0$, so that for all $0<h<h_0$, we have $e^h-1\geq \f{1}{2} h$. We can, without loss of generality take $h_0\leq 1$. 
   
   We split the integration in \eqref{711} in two intervals $s\in (\tau-h_0, \tau)$ and $s\in (0, \tau-h_0)$.  For the latter, we have that $e^{\tau} - e^s\geq e^\tau-e^{\tau-h_0}=e^\tau(1-e^{-h_0})$. So, 
   \begin{eqnarray*}
    	\int_0^{\tau-h_0} e^{b (\tau- s)} e^{- \de (e^{\tau} - e^s)} e^{- c s} ds \leq 	
   e^{-  \de(1-e^{-h_0})  e^{\tau}}  \int_0^{\tau-h_0} e^{b (\tau- s)}   ds \leq   e^{-  \de(1-e^{-h_0})  e^{\tau}} 
   e^{|b|\tau} \tau \leq C_{b, \de} e^{- (c+ 1) \tau},
   \end{eqnarray*}
   where we obtain a much better, exponential in $e^\tau$, decay rate. 
  For the case $s\in (\tau-h_0, \tau)$, observe first that by the choice of $h_0$, we have 
   $$
   e^\tau-e^s=e^s(e^{\tau-s}-1)\geq \f{1}{2} e^s(\tau-s) \geq \f{1}{8} e^{\tau}(\tau-s). 
   $$
  We need to control $e^{- c \tau}   	\int_{\tau-h_0}^\tau   e^{- \f{\de}{8} e^{\tau}(\tau-s)}  ds $, as follows 
  \begin{eqnarray*}
 e^{- c \tau}   	\int_{\tau-h_0}^\tau   e^{- \f{\de}{8} e^{\tau}(\tau-s)}  ds \leq  e^{- c \tau}    \int_0^{1} 
  e^{- \f{\de}{8} e^{\tau}s}  ds  \leq  8 e^{- (c+1) \tau}  \int_0^\infty e^{-\de z} dz = \f{8}{\de} 
   e^{- (c+1) \tau}. 
    \end{eqnarray*}

\end{document}